\documentclass[12pt]{amsart}
\usepackage{fullpage}
\usepackage[utf8]{inputenc}
\usepackage{graphicx}
\usepackage{enumerate}
\usepackage{microtype}
\usepackage{xcolor}
\usepackage[colorlinks,bookmarks,linkcolor=black,citecolor=black]{hyperref} 
\usepackage{textcase}
\usepackage{mathtools}
\usepackage{bm}

\mathtoolsset{showonlyrefs}

\usepackage[mathlines,switch]{lineno}

\newtheorem{theorem}{Theorem}[section]
\newtheorem*{theorem*}{Theorem}
\newtheorem*{proposition*}{Propostion}
\newtheorem{lemma}[theorem]{Lemma}
\newtheorem{conjecture}[theorem]{Conjecture}
\newtheorem{proposition}[theorem]{Proposition}
\newtheorem{claim}[theorem]{Claim}
\newtheorem{corollary}[theorem]{Corollary}
\theoremstyle{definition}
\newtheorem{definition}[theorem]{Definition}
\newtheorem{observation}[theorem]{Observation}

\newtheorem{fact}[theorem]{Fact}

\theoremstyle{plain}
\newtheorem*{lem:embed}{Lemma \ref{embed}}
\newtheorem*{lem:embed2}{Lemma \ref{embed2}}

\theoremstyle{remark}

\numberwithin{equation}{section}

\newcommand{\EMAIL}[1]{  \textit{E-mail}: \texttt{#1}}

\newcommand{\abs}[1]{\lvert#1\rvert}
\newcommand{\norm}[1]{\lVert#1\rVert}
\let\eps\varepsilon
\let\del\delta

\title[Ramsey numbers via non-linear optimisation (\today)]{Exact Ramsey numbers of odd cycles via nonlinear optimisation}

\author{Matthew Jenssen\dag, Jozef Skokan\dag\ddag}

\thanks{
\dag Department of Mathematics, London School of Economics, Houghton Street,
    London WC2A 2AE, U.K.
\EMAIL{\{m.o.jenssen|j.skokan\}@lse.ac.uk}}

\thanks{
\ddag
Department of Mathematics, University of Illinois at Urbana-Champaign, 1409 W. Green Street, Urbana, IL 61801, USA}

\date{\today}

\dedicatory{}


\begin{document}

\begin{abstract}
For a graph $G$, the $k$-colour Ramsey number $R_k(G)$ is the least integer $N$ such that every $k$-colouring of the edges of the complete graph $K_N$ contains a monochromatic copy of $G$. Let $C_n$ denote the cycle on $n$ vertices.
We show that for fixed $k\geq2$ and $n$ odd and sufficiently large,
$$
R_k(C_n)=2^{k-1}(n-1)+1.
$$
This resolves a conjecture of Bondy and Erd\H{o}s [J. Combin. Th. Ser. B \textbf{14} (1973), 46--54] for large $n$. The proof is analytic in nature, the first step of which is to use the regularity method to relate this problem in Ramsey theory to one in nonlinear optimisation.  This allows us to prove a stability-type generalisation of the above and establish a surprising correspondence between extremal $k$-colourings for this problem and perfect matchings in the $k$-dimensional hypercube $Q_k$.
\end{abstract}

\maketitle

\section{Introduction}
\thispagestyle{empty}
One of the most well-known and extensively researched problems in combinatorics is that of~determining the \emph{Ramsey numbers} of graphs, defined as follows. Given graphs $G_1, G_2, \ldots, G_k$, the \emph{Ramsey number} $R(G_1, \ldots, G_k)$ is the least integer $N$ such that any colouring of the edges of the complete graph $K_N$ on $N$ vertices with $k$ colours contains a monochromatic copy of $G_i$ in the $i$-th colour for some $i$, $1\leq i\leq k$. In the case where $G_1, \ldots, G_k$ are all isomorphic to the graph $G$, we call $R(G_1,\ldots, G_k)$ the \emph{$k$-colour Ramsey number} of $G$ and denote it by $R_k(G)$. Broadly speaking the philosophy underpinning Ramsey theory is that large, potentially highly disordered structures must contain ordered substructures.

Ramsey theory owes its name to the seminal paper of Frank Ramsey~\cite{Ram} where it is shown that Ramsey numbers are finite. The oldest and most famous examples of Ramsey numbers are those involving cliques. The systematic study of such Ramsey numbers began with a paper of Erd\H{o}s and Szekeres~\cite{ErdSze} who considered the problem of determining $R_2(K_k)$, where $K_k$ denotes the clique on $k$ vertices. Erd\H{o}s and Szekeres~\cite{ErdSze} and Erd\H{o}s~\cite{ErdL} showed that $2^{k/2}\leq R_2(K_k)\leq 4^k$. The problem of improving these bounds has gained significant notoriety and only small improvements have been made over the last eighty years. In the case where the graphs in question are sparse in some sense (e.g. they have bounded maximum degree), the situation seems somewhat simpler. The cycle on $n$ vertices $C_n$ was one of the earliest subjects in the study of Ramsey numbers of~sparse graphs. The behaviour of the Ramsey number $R(C_{n_1}, C_{n_2})$ has been studied and fully determined by several authors, including Bondy and Erd\H{o}s~\cite{BonErd}, Faudree and Schelp~\cite{FauSch}, and Rosta~\cite{Rosta}. For example it is known that
 \[ R_2(C_n) 
    =
  \begin{cases}2n-1, & \text{if $n\ge5$ is odd,}\\
  \frac{3n}{2} -1, & \text{if $n\ge6$ is even.}
\end{cases}
 \]

Results such as this one that exactly determine $R(G_1, G_2)$ for a pair of graphs $G_1, G_2$ are by now fairly plentiful. See Radziszowski~\cite{survey} for an excellent survey of such results. However, in the case where more than two colours are involved such results are still rather rare. The only non-trivial class of graphs for which the $k$-colour Ramsey number is exactly determined for arbitrary $k$ is that of matchings (a result due to Cockayne and Lorimer~\cite{lorimer}). In this paper we address the following conjecture attributed to Bondy and Erd\H{o}s~\cite{BonErd}.
\begin{conjecture}\label{bon}
If $k\geq2$ and $n>3$ is odd then 
 $$R_k(C_n)=2^{k-1}(n-1)+1.$$
\end{conjecture}
Note that the conjecture deals specifically with the case where $n$ is odd. Odd and even cycles behave rather differently in this context due to the fact that an even cycle is bipartite whereas an odd cycle is not. We mention that Erd\H{o}s and Graham~\cite{ErdGra} proved the bounds
\begin{equation}\label{Erd}
2^{k-1}(n-1)+1\leq R_k(C_n)\leq (k+2)!n,
\end{equation}
for all $k\geq2$ and all odd $n>3$. In this paper we prove the following. 

\begin{theorem}\label{main}
For any fixed $k\geq2$ and odd $n$ sufficiently large,
$$R_k(C_n)=2^{k-1}(n-1)+1.$$
\end{theorem}
We therefore resolve Conjecture \ref{bon} for large $n$. We will in fact prove a stability-type strengthening of this result (see Theorem \ref{conj} below). Very recently Day and Johnson~\cite{DayJoh} showed that in the opposite regime, where we fix an odd $n$ and let $k$ be sufficiently large, one in fact has $R_k(C_n)>(n-1)(2+\eps)^{k-1}$ for some $\eps=\eps(n)>0$, and so Conjecture \ref{bon} is false when $n$ is small with respect to $k$. The qualification that $n$ is sufficiently large in Theorem~\ref{main} is therefore necessary, however due the use of compactness arguments in the proof, we obtain no effective bound on how large $n$ must be with respect to $k$.

In view of Theorem \ref{main}, let us call a $k$-colouring of the complete graph on $2^{k-1}(n-1)$ vertices which does not contain a monochromatic copy of $C_n$ an \emph{extremal $k$-colouring}. The lower bound in \eqref{Erd} was established by observing that one can naturally construct extremal $k$-colourings by induction. Indeed if there exists a $k$-colouring of the edges of the complete graph $K_m$ with no monochromatic $C_n$, then by joining two such copies of $K_m$ by edges of colour $k+1$, one obtains a $(k+1)$-colouring of $K_{2m}$ with no monochromatic $C_n$. The base construction, for $k=1$, is simply a monochromatic clique of size $n-1$. It was believed that all extremal $k$-colourings come from such a doubling argument. We show that this is not the case, providing a classification of extremal $k$-colourings which exposes a surprising correspondence between extremal $k$-colourings and perfect matchings in the $k$-dimensional hypercube $Q_k$. 

 The first breakthrough towards Conjecture \ref{Erd} was made by \L uczak~\cite{Lucz} who used the regularity method to show that the $k=3$ case holds asymptotically i.e. that for $n$ odd,
$$
R_3(C_n)=4n+o(n)\ \text{as}\ n\to\infty.
$$
\L uczak's method of applying regularity in this setting has proven extremely fruitful (see e.g. \cite{FigLucz,Gya, Exact, Lucz, Multi}) and has since become a standard tool. We will come to describe the method in more detail as it provides the starting point for the present paper.
 
Building on \L uczak's ideas, Kohayakawa, Simonovits and Skokan~\cite{Exact} paired the regularity method with stability arguments to resolve Conjecture \ref{bon} for $k=3$ and $n$ large. The case where $k\geq4$ has since remained open. Progress was made by \L uczak, Simonovits and Skokan~\cite{Multi} who showed that for $k\geq4$ and odd $n$,
$$
R_k(C_n)\leq k2^kn+o(n)\ \text{as}\ n\to\infty.
$$

To conclude this section we give a broad overview of the proof method of Theorem~\ref{main}. Let $\mathcal{G}_n$ denote the (finite) class of all cliques that admit a $k$-colouring with no monochromatic copy of $C_n$. Determining $R_k(C_n)$ is then equivalent to determining the maximum $N$ such that $K_N\in \mathcal{G}_n$. Using the regularity method, we relate this problem to finding the maximum $\ell_1$-norm of an element in a certain compact subset $\mathcal{S}$ of $\mathbb{R}^{3^k}$. This allows us to import analytic and topological tools in support of our proof. The relation is such that maximal elements of $\mathcal{S}$ correspond to extremal $k$-colourings for Theorem \ref{main}. Moreover by classifying the extremal points of $\mathcal{S}$ we can classify the extremal $k$-colourings and prove a stability type strengthening of Theorem~\ref{main}, generalising the main result from~\cite{Exact}. We show that each perfect matching of the hypercube $Q_k$ gives rise to a class of extremal $k$-colourings. On the other hand, any extremal $k$-colouring must be `close' to one such construction. We defer precise statements to the Section \ref{decompsec}. The number of essentially different classes of extremal $k$-colourings is equal to the number of equivalence classes of perfect matchings in $Q_k$ with respect to its automorphism group and this number is doubly exponential in $k$. Such a plethora of extremal constructions is usually forbidding when trying to establish stability type results, we believe that the fact we can overcome this obstacle is largely down to the analytic perspective.

\section{Notation and Terminology}
Let us collect some notation that we use throughout the paper. For $k\in \mathbb{N}$, we let $[k]$ denote the set $\{1, \ldots, k\}$. For a set $S$, we let $\binom{S}{2}$ denote the set of all unordered pairs of distinct elements of $S$.

All graphs considered here will be finite. For a graph $G=(V,E)$, we let $v(G)=\abs{V}$ and $e(G)=\abs{E}$. For $v\in V$, we let $N_G(v)$ denote the neighbourhood of $v$ in $G$ and let $\delta(G)$ denote the minimum degree of $G$. For $X\subseteq V$ we denote by $G[X]$ the subgraph of $G$ induced by the vertices of $X$. For disjoint subsets $A, B\subseteq V$, we denote by $G[A,B]$ the bipartite graph with vertex set $A\cup B$ and edge set $\{\{a,b\}\in E: a\in A, b\in B\}$, and we let $e_G(A,B)$ denote the size of this set. In the case where $A=\{v\}$ a singleton, we write $G[v,B]$ instead of $G[\{v\}, B]$. All subscripts in the above notation may be suppressed if they are clear from the context. At times we may slightly abuse notation by writing $v\in G$ and $\{x,y\}\in G$ in lieu of $v\in V(G)$ and $\{x,y\}\in E(G)$ respectively. 

Let $W=w_0w_1\ldots w_\ell$ be a walk in $G$ (that is a sequence of vertices $w_0,w_1,\ldots w_\ell$ such that $\{w_i,w_{i+1}\}$ is an edge of $G$ for all $i<\ell$). If all of the $w_i$ are distinct then we call $W$ a \emph{path of length $\ell$}. We may also refer to $W$ as a $w_0w_\ell$-path to distinguish its endpoints. If all the $w_i$ are distinct except $w_0=w_\ell$ then we call $W$ a \emph{cycle of length $\ell$}. We will also concatenate walks in the natural way. For example if $U=u_0\ldots u_m$ is a walk in $G$ such that $u_m=w_0$, we let $UW$ denote the walk $u_0\ldots u_mw_1\ldots w_\ell$. If $x$ is a vertex such that $\{w_\ell,x\}$ is an edge of $G$ then we let $Wx$ denote the walk $w_0\ldots w_\ell x$.

A \emph{$k$-coloured graph} is a graph $G=(V,E)$ equipped with some function $\varphi: [k]\to E$. Furthermore, for $i\in[k]$, we let $G_i$ denote the subgraph $(V, \varphi^{-1}\{i\})$ of $G$. We call $G_i$ the $i$th \emph{colour class} of $G$. For a digraph $F$ and vertex $v\in V(F)$, we let $d^{-}(v)$, $d^{+}(v)$ denote the indegree and outdegree of $v$ respectively.

For $x\in\mathbb{R}^d$ we let $\norm{x}$ denote the $\ell_1$-norm of $x$ i.e. $\norm{x}=\sum_{i=1}^d\abs{x_i}$. Furthermore, given $\eps>0$, we let $B_\eps(x):=\{z\in\mathbb{R}^d: \norm{z-x_0}<\eps\}$, the open ball of radius $\eps$ centred at $x$. We let $supp(x)$ denote the support of $x$. 
 
 In the statements of theorems and lemmas it will be useful to use the notation $\alpha\ll\beta$ to mean that there is an increasing function $\alpha(x)$ so that the statement is valid for $0<\alpha<\alpha(\beta)$. When we need to refer to this function at a later stage, we include the number of the lemma (or theorem) the function appears in as a subscript. For example, $\delta_{\ref{embed}}(x)$ denotes the implied function $\delta(x)$ from Lemma~\ref{embed}. Throughout the paper we omit the use of floor and ceiling symbols where they are not crucial.

\section{A Graph Decomposition, Extremal Colourings and Stability}\label{decompsec}
In this section we describe the extremal colourings and give precise statements of the stability results referred to in the Introduction. We also introduce some key concepts and results that will be used throughout the paper and give a more detailed overview of our proof methods. We begin by introducing a way of decomposing an arbitrary $k$-coloured graph. This decomposition will play a central role for us and is similar to a decomposition introduced in~\cite{Multi}. 

\subsection{A Graph Decomposition}\label{sec:decomp}
Let $G$ be a $k$-coloured graph. For each $i\in[k]$, we write $G_i=G_i'\cup G_i''$, where $G_i'$ is the union of the bipartite components of $G_i$ and $G_i''$ is the union of the non-bipartite components of $G_i$. For each $i\in[k]$, write $V(G_i')=V_0^i\cup V_1^i$ where $V_0^i$ and $V_1^i$ are the vertex classes of a bipartition of $G_i'$ and set $V_\ast^i=V(G_i'')$. For $\tau\in\{0,1,\ast\}^k$, let $V_\tau=\bigcap_{j=1}^k V_{\tau_j}^j$ and note that
\begin{equation}\label{union}
V(G)=\bigcup_{\tau\in\{0,1,\ast\}^k}V_\tau,\ \ \mbox{a disjoint union.}
\end{equation}
We call $(V_\tau: \tau\in\{0,1,\ast\}^k)$ a $\emph{profile partition}$ of $G$ and we call the corresponding vector $(\abs{V_\tau}: \tau\in\{0,1,\ast\}^k)$ a $\emph{profile}$ of $G$. We will often denote a profile of $G$ by $x(G)$. Note that $G$ may admit multiple profile partitions since we made an arbitrary choice in choosing the bipartition $V(G_i')=V_0^i\cup V_1^i$ for each $i\in[k]$.

\subsection{Extremal Colourings and the Hypercube}
For $k\in \mathbb{N}$, we let $Q_k$ denote the $k$-dimensional hypercube i.e. the graph on vertex set $\{0,1\}^k$ and edge set consisting of pairs differing in exactly one coordinate. It will be useful to think of an element $\tau\in\{0,1,\ast\}^k$ as a subcube of the $k$-dimensional hypercube $Q_k$ via the correspondence
$$\tau\leftrightarrow Q(\tau):=\{c\in\{0,1\}^k: c_j=\tau_j\  \mbox{if}\ \tau_j\in\{0,1\}\}.$$
In other words we think of a coordinate $j$ where $\tau_j=\ast$ as a `missing bit' and let $Q(\tau)$ be the set of all possible ways of filling in these bits. For example, 
\begin{center}
if $k=3$ and $\tau=(0,\ast,\ast)$, then $Q(\tau)=\{(0,0,0), (0,0,1), (0,1,0), (0,1,1)\}$.
\end{center}
We define the \emph{weight} of $\tau$ to be the size of the set $\{i\in[k]: \tau_i=\ast\}$ (i.e. the number of missing bits) and denote it by $\omega(\tau)$. Note that $\abs{Q(\tau)}=2^{\omega(\tau)}$. In the language of the hypercube, $\omega(\tau)$ is the dimension of the subcube $Q(\tau)$. In particular if $\omega(\tau)=1$, then we think of $Q(\tau)$ as an \emph{edge} of $Q_k$.
 
We can now describe a class of extremal $k$-colourings in terms of perfect matchings in $Q_k$. Let $\mathcal{M}$ be a perfect matching of $Q_k$. We express each edge of $\mathcal{M}$ as an element (of weight 1) of $\{0,1,\ast\}^k$. Let $G=K_N$ where $N=2^{k-1}(n-1)$ and let $V(G)=\bigcup_{\tau\in\mathcal{M}}V_\tau$ be a partition of $V(G)$ where $\abs{V_\tau}=n-1$ for all $\tau\in\mathcal{M}$. For each $\tau\in \mathcal{M}$, colour all edges in $G[V_\tau]$ with the colour $i$, where $i$ is the coordinate for which $\tau_i=\ast$. For $\tau,\sigma\in\mathcal{M}$, arbitrarily colour the edges between $V_\tau$ and $V_\sigma$ with any colour $j$ for which $\{\sigma_j, \tau_j\}=\{0,1\}$ (i.e. edges $\tau,\sigma$ lie in opposite subcubes of $Q_k$ of codimension 1 separated by the $j$th coordinate). It follows that each colour class of such a colouring is the disjoint union of cliques of size $n-1$ and a bipartite graph and therefore contains no monochromatic copy of $C_n$. We call such a colouring a \emph{hypercube colouring with clique size $n-1$}.

If we inductively construct a perfect matching on $Q_k$ by taking two perfect matchings on a disjoint pair of subcubes of codimension 1 and consider the associated hypercube colouring, we recover the inductive colourings of Erd\H{o}s and Graham~\cite{ErdGra} described in the Introduction. However for $k\geq4$, not all perfect matchings of $Q_k$ decompose as the union of two matchings on a pair of codimension 1 subcubes, and so we obtain some genuinely new colourings. In particular, a novel feature that appears for $k\geq4$ colours is that there exist extremal $k$-colourings that contain a monochromatic cliques of size $n-1$ in all $k$ possible colours.
 
\subsection{Stability} 
In this subsection we state a theorem to the effect that the hypercube colourings considered in the previous subsection are the only extremal $k$-colourings for our problem. Moreover we assert that almost extremal colourings are in some sense `close' to a~hypercube colouring. Let us make this more precise. 

\begin{definition}\label{def:close}
Let $G$ and $H$ be $k$-coloured graphs with $V(H)\subseteq V(G)$. Let $\eps>0$, then we say that $G$ is \emph{$\eps$-close} to $H$ if $\abs{G_i\triangle H_i}\leq \eps v(G)^2$ for all $i\in[k]$.
\end{definition}
Informally we may say that $G$ and $H$ as above are close in \emph{edit distance}. We may now state the main result of this paper.
\begin{theorem}\label{conj}
Let $k\geq2$, let $\frac{1}{n}\ll\eta\ll\eps\ll1$, where $n$ is odd, and let $N>(2^{k-1}-\eta)n$. Then if $G=K_N$ is $k$-coloured with no monochromatic copy of $C_n$, then $N\leq2^{k-1}(n-1)$ and there exists a hypercube colouring $H$ such that $G$ is $\eps$-close to $H$.
\end{theorem} 
 Note that Theorem~\ref{main} follows as an immediate corollary. The $k=3$ case of Theorem~\ref{conj} was proved in \cite{Exact} where the two classes of colourings the authors consider can be viewed as the colourings that arise from the two isomorphism classes of perfect matchings in~$Q_3$. An~interesting feature of Theorem~\ref{conj} is that it deals with a wide variety of extremal colourings. Indeed if $\mathcal{M}_1$ and $\mathcal{M}_2$ are perfect matchings of $Q_k$ that lie in distinct equivalence classes under the action of the automorphism group of $Q_k$, then it's not difficult to show that there are hypercube colourings associated to $\mathcal{M}_1$ that are not isomorphic to any hypercube colouring associated to $\mathcal{M}_2$. It is also interesting to note that even though we can prove a~stability statement around hypercube colourings, the structure of these colourings is not well understood. This is simply due to the fact that the structure of perfect matchings in the hypercube is not well understood. Indeed, even enumerating the perfect matchings (or their equivalence classes) in~$Q_k$ is a~well-studied and difficult problem. Let $f(k)$ be the number of equivalence classes of perfect matchings in $Q_k$. It is clear that $f(3)=2$ and so we obtain two essentially different extremal $3$-colourings as in~\cite{Exact}. Graham and Harary~\cite{GraHar} showed that f(4)=8 and recently {\"O}sterg{\aa}rd and Pettersson~\cite{OstPet} determined (with a large amount of computer time) $f(5)$, $f(6)$ and $f(7)$. The function $f(k)$ grows rather rapidly; it is amusing to note that already $f(7)=607158046495120886820621$ and so we have this many essentially different classes of extremal $7$-colourings. It was shown in~\cite{clark1997number} that the number of perfect matchings in $Q_k$ is $[(1+o(1))k/e]^{2^{k-1}}$ (although this result in fact follows from a theorem in~\cite[p.312]{plummer1986matching}). Since the automorphism group of $Q_k$ has size $k!2^k$ it follows that $f(k)=[(1+o(1))k/e]^{2^{k-1}}$ also.
 
\subsection{Proof of Theorem~\ref{conj}: An Overview}
The regularity method, discussed briefly in the introduction, plays a central role in our proof method. We include an~informal discussion of the method here, deferring details until later. We start with a~definition.

\begin{definition}\label{def:conmat}
Let $F$ be a connected graph whose largest matching saturates $m$ vertices, then we call $F$ a \emph{connected matching of order $m$}. We distinguish a particular matching of~largest size $M_F$ in $F$ and refer to an edge of $M_F$ as a \emph{matching edge} of $F$. If in addition $F$ is non-bipartite, we call $F$ an \emph{odd} connected matching of order $m$. 
\end{definition}

The idea behind the regularity method is as follows. Suppose that $G$ is a $k$-coloured complete graph on $N$ vertices. Let $G_1,\ldots, G_k$ be its colour classes. We apply the multicolour version of the Regularity Lemma~\cite{Szem} and obtain a regular partition of the vertex set $V(G)$ into $t+1$ classes $V(G)= V_0\cup \ldots\cup V_t$.
We construct an auxiliary graph $R$ with vertex set ${1,...,t}$ and the edge
set formed by pairs $\{i, j\}$ for which $(V_i, V_j )$ is regular with respect to $G_1, \ldots, G_k$. We colour each edge $\{i,j\}$ in $R$ by the majority colour in the pair $(V_i, V_j )$. The crucial point is that if $R$ contains a monochromatic odd connected matching of order greater than $m$, then $G$ contains a monochromatic cycle $C_\ell$ where $\ell$ can take essentially any odd value smaller than $mN/t$. It follows that if $G$ contains no monochromatic copy of $C_n$, then $R$ cannot contain a monochromatic odd connected matching of order larger than $nt/N$. The advantage of this perspective is that forbidding a large connected matching is far more restrictive than forbidding a cycle of a given length. Indeed a cycle is itself an example of a~connected matching, and so if a graph contains no connected matching of order greater than $m$ then it contains no cycle of length greater than $m$. The following theorem of Erd\H{os} and Gallai~\cite{EG} shows that this is a very strict condition.
\begin{theorem}\label{EG}
Let $m\geq3$. If $G$ is a graph which contains no cycle of length greater than $m$, then $e(G)\leq m(v(G)-1)/2$. 
\end{theorem}

 The price one pays is that $R$ is not a complete graph, however it can be chosen to be as dense as one likes. We are now able to state a theorem that is a major stepping stone toward the proof of Theorem~\ref{conj}.

\begin{theorem}\label{reduced}
Let $k\geq2$ and let $\frac{1}{n}\ll\del\ll\eps\ll1$, where $n$ is odd. If $G$ is a $k$-coloured graph with $v(G)=2^{k-1}n$ and $e(G)\geq(1-\delta)\binom{v(G)}{2}$ containing no monochromatic odd connected matching of order $\geq(1+\delta)n$, then for any choice of profile $x(G)$ of $G$, there exists a hypercube colouring $H$ with profile $x(H)$ satisfying
$$
\|x(G)-x(H)\|\leq\varepsilon n.
$$  
\end{theorem}

The proof of Theorem~\ref{reduced} occupies the majority of this paper. In the final section we show how Theorem~\ref{conj} follows from \ref{reduced} via combinatorial stability arguments and the regularity method. The outline of the proof of Theorem~\ref{reduced} is as follows. Let $G$ be as in the statement of Theorem~\ref{reduced} and let $x(G)$ denote a profile of $G$. Our starting point is to translate the combinatorial constraint of containing no large monochromatic odd connected matching into an analytic constraint on $x(G)$ of the form
\begin{equation}\label{quad}
F(x(G))\leq0,
\end{equation}
where $F$ is a quadratic form which we derive in the next section. We then view \eqref{quad} as a~constraint in an optimisation problem where we wish to maximise the objective function~$\norm{x(G)}$. Recalling that $\norm{x(G)}=v(G)$, we get a corresponding upper bound on the order of $G$. It turns out that optimal solutions to this optimisation problem correspond to the profiles of~hypercube colourings. Solving the optimisation problem is the subject of Sections~\ref{sec:compress} and~\ref{sec:convex}. In Section~\ref{sec:stab} we use compactness arguments to show that almost optimal solutions must be close in $\ell_1$-norm to the profile of a hypercube colouring. We then translate this analytic stability into the, more combinatorial, stability statement of Theorem~\ref{reduced}. Note that Theorem~\ref{reduced} will be applied to a reduced graph like the one described above. The focus of the final section is to show that if the profile of this reduced graph is close in $\ell_1$-norm to the profile of a hypercube colouring, then the original graph is close in edit distance to a~hypercube colouring.

\section{Deriving the Analytic Constraints}
Given a $k$-coloured graph $G$, we will show how to translate the combinatorial constraint of containing no large monochromatic odd connected matching into an analytic constraint on the profile of $G$.

From here on, throughout the paper, we let $k\geq2$ be a fixed integer. Let $G$ be a $k$-coloured graph. First we distinguish between two types of edges of $G$. If $e\in E(G)$ is coloured with the colour $j$ and lies in a bipartite component of $G_j$ then we call $e$ a \emph{bipartite} edge. We call $e$ \emph{non-bipartite} otherwise. Let $(V_\tau: \tau\in\{0,1,\ast\}^k)$ be a profile partition of $G$. We make two simple observations regarding the profile partition of a $k$-coloured graph.
\begin{observation}\label{obs1} If $e\in E(G)$ is a bipartite edge of colour $j$ then it must have endpoints in parts $V_\tau, V_\sigma$ for some $\tau, \sigma\in \{0,1,\ast\}^k$ such that $\tau_j=0$ and $\sigma_j=1$.
\end{observation}
\begin{observation}\label{obs2} If $e\in E(G)$ is a non-bipartite edge of colour $j$ then it must have endpoints in parts $V_\tau, V_\sigma$ for some (not necessarily distinct) $\tau, \sigma\in \{0,1,\ast\}^k$ such that $\tau_j=\sigma_j=\ast$.
\end{observation}
This motivates the following definitions.
\begin{definition}
We say that $\sigma, \tau \in \{0,1,\ast\}^k$ are \emph{distinguishable} if $\{\sigma_j,\tau_j\}=\{0,1\}$ for some $j\in[k]$. We say that $\sigma$ and $\tau$ are \emph{indistinguishable} otherwise.
\end{definition}
\begin{definition}
If $\sigma, \tau\in\{0,1,\ast\}^k$ are such that either $(i)$ $\sigma,\tau$ are distinguishable or $(ii)$ $\sigma_j=\tau_j=\ast$ for some $j\in[k]$, then we say that $\sigma$ and $\tau$ are \emph{compatible}. We say that $\sigma, \tau$ are \emph{incompatible} otherwise.
\end{definition}
Viewing elements of $\{0,1,\ast\}^k$ as subcubes of $Q_k$, we may reinterpret these definitions as follows.

\begin{lemma}\label{distinguishable}
 Let $\sigma, \tau\in\{0,1,\ast\}^k$. Then $\sigma, \tau$ are distinguishable if and only if $Q(\tau)\cap Q(\sigma)=\emptyset$. Furthermore,  $\sigma, \tau$ are incompatible if and only if $\abs{Q(\tau)\cap Q(\sigma)}=1$.
 \end{lemma}
\begin{proof}
By the definition of the sets $Q(\tau), Q(\sigma)$ we have
$$
Q(\tau)\cap Q(\sigma)=\{c\in\{0,1\}^k: c_j=\tau_j\ \text{if}\ \tau_j\in\{0,1\}\ \text{and}\  c_j=\sigma_j\ \text{if}\  \sigma_j \in \{0,1\}\}.
$$
This is empty if and only if there exists a $j\in[k]$ such that $\sigma_j, \tau_j\in\{0,1\}$ and $\sigma_j\neq\tau_j$ i.e. if and only if  $\sigma, \tau$ are distinguishable. Let $T=\{i\in [k]: \sigma_i=\tau_i=\ast\}$. If $\sigma, \tau$ are indistinguishable then we see that $\abs{Q(\tau)\cap Q(\sigma)}=2^{\abs{T}}$. Therefore $\abs{Q(\tau)\cap Q(\sigma)}=1$ if and only if $\sigma, \tau$ are indistinguishable and $T=\emptyset$ i.e. $\sigma, \tau$ are incompatible. 
\end{proof}

From now on, we let
\[
D=\left\{\{\sigma,\tau\}\in\binom{\{0,1,\ast\}^k}{2}:\ \sigma,\tau\ \mbox{are distinguishable}\right\}.
\]
It will also be convenient to make the following definition.

\begin{definition}
Let $\alpha>0$ and let $G$ be a graph such that $e(G)\geq\alpha\binom{v(G)}{2}$. Then we say that $G$ is \emph{$\alpha$-dense}.
\end{definition}
This next proposition provides the link between our combinatorial problem and a problem in nonlinear optimisation.
\begin{proposition}\label{constraints}
Let $C>1$, $0<\del<1$ and let $n>1/\del$. Suppose that $G$ is a $(1-\del)$-dense, $k$-coloured graph with $v(G)=Cn$, containing no monochromatic odd connected matching of order $\geq(1+\del)n$. Let $x$ be a profile of $G$ and let $v=x/n$. Then the following hold:
\begin{enumerate}
\item $$\left(\sum_{\tau\in\{0,1,\ast\}^k}v_\tau\right)^2-2\sum_{\{\sigma, \tau\}\in D}v_\sigma v_\tau-\sum_{\tau\in\{0,1,\ast\}^k}\omega(\tau)v_\tau\leq\del k C^2.$$\\
\item $v_\tau\leq1+2\sqrt{\del}C$ whenever $\omega(\tau)=1$.\\
\item $v_\tau v_\sigma\leq2\del C^2$ whenever $\sigma$ and $\tau$ are incompatible.
\end{enumerate}
\end{proposition}

\begin{proof}
Let us first remind ourselves of the graph decomposition discussed in Subsection \ref{sec:decomp}. For each $i\in[k]$, we write $G_i=G_i'\cup G_i''$, where $G_i'$ is the union of the bipartite components of $G_i$ and $G_i''$ is the union of the non-bipartite components of $G_i$. For each $i\in[k]$, write $V(G_i')=V_0^i\cup V_1^i$ where $V_0^i$ and $V_1^i$ are the vertex classes of a bipartition of $G_i'$ and set $V_\ast^i=V(G_i'')$. For $\tau\in \{0,1,\ast\}^k$, set $V_\tau=\bigcap_{j=1}^k V_{\tau_j}^j$. Let $x=(\abs{V_\tau}: \tau\in\{0,1,\ast\}^k)$ be the profile corresponding to this partition. Let $N=v(G)$ and note that
\begin{equation}\label{deco}
N=\sum_{\tau\in\{0,1,\ast\}^k} x_\tau.
\end{equation}
It follows from Observation \ref{obs1} that the number of bipartite edges in $G$ is at most
$
\sum_{\{\sigma, \tau\}\in D}x_\sigma x_\tau.
$
Letting $e_0$ denote the number of non-bipartite edges in $G$ we therefore have that
\begin{equation}\label{non1}
e_0\geq e(G)-\sum_{\{\sigma, \tau\}\in D} x_\sigma x_\tau.
\end{equation}
Since $N\geq1/\delta$, we have
\begin{equation}\label{e0}
e(G)\geq(1-\delta)\binom{N}{2}\geq(1-2\delta)\frac{N^2}{2}.
\end{equation}
Combining \eqref{deco}, \eqref{non1} and \eqref{e0} gives
\begin{equation}\label{non2}
e_0\geq \frac{1}{2}\left(\sum_{\tau\in\{0,1,\ast\}^k}x_\tau\right)^2-\sum_{\{\sigma, \tau\}\in D}x_\sigma x_\tau-\delta N^2.
\end{equation}
We now find a corresponding upper bound for $e_0$. Recall that for $\tau\in\{0,1,\ast\}^k$, the weight $\omega(\tau)$ of $\tau$ is defined to be the size of the set $\{i\in[k]:\tau_i=\ast\}$.

By assumption, for each $i\in[k]$, every connected component of $G_i''$ has no matching on $(1+\delta)n$ vertices and so in particular $G_i''$ has no cycle of length greater than $(1+\delta)n$. Theorem~\ref{EG} therefore implies that
\begin{equation}\label{sparse}
e(G_i'')\leq (1+\delta)\frac{n}{2}\abs{V_\ast^i}.
\end{equation}
Observe that 
\begin{equation}\label{local}
\abs{V_\ast^i}=\sum_{\{\tau\in\{0,1,\ast\}^k : \tau_i=\ast\}}x_{\tau}.
\end{equation}
Since each non-bipartite edge of $G$ belongs to $E(G_i'')$ for some $i$, \eqref{sparse} and \eqref{local} provide the upper bound
\begin{equation}\label{above}
e_0\leq\sum_{i=1}^ke(G_i'')\leq(1+\delta)\frac{n}{2}\sum_{i=1}^k\abs{V_\ast^i}=(1+\delta)\frac{n}{2}\sum_{\tau\in\{0,1,\ast\}^k}\omega(\tau)x_{\tau}.
\end{equation}
Since $\omega(\tau)\leq k$ for all $\tau\in \{0,1,\ast\}^k$ by definition, \eqref{deco} and \eqref{above} imply the bound
\begin{equation}\label{above2}
e_0\leq\frac{1}{2}\delta nkN+\frac{n}{2}\sum_{\tau\in\{0,1,\ast\}^k}\omega(\tau)x_{\tau}.
\end{equation}

Recall that $v=x/n$. Comparing the bounds \eqref{non2} and \eqref{above2} and scaling the resulting inequality by $2/n^2$ yields

$$
\left(\sum_{\tau\in\{0,1,\ast\}^k}v_\tau\right)^2-2\sum_{\{\sigma, \tau\}\in D}v_\sigma v_\tau-2\del C^2\leq\sum_{\tau\in\{0,1,\ast\}^k}\omega(\tau)v_\tau+\del kC.
$$
This establishes (1). Notice that if $\tau\in\{0,1,\ast\}^k$ is such that $\omega(\tau)=1$, then $G[V_\tau]$ is monochromatic with all edges non-bipartite by Observations \ref{obs1} and \ref{obs2}. $G[V_\tau]$ therefore contains no cycle of length greater than $(1+\delta)n$ and so by Theorem~\ref{EG} and the fact that $G$ has at most $\delta\binom{N}{2}$ edges missing
$$
\binom{x_\tau}{2}-\delta\binom{N}{2}\leq e(G[V_\tau])\leq (1+\delta)\frac{n}{2}x_\tau.
$$
It follows that
$$
v_\tau^2\leq (1+2\delta)v_\tau+\delta C^2,
$$
from which (2) follows.
Finally, let us note that by Observations \ref{obs1} and \ref{obs2}, if $\sigma, \tau$ are incompatible, then there can be no edges lying between $V_\sigma$ and $V_\tau$. Since $G$ has at most $\delta\binom{N}{2}$ edges missing we must then have
$$
x_\tau x_\sigma\leq2\delta N^2
$$
(note that this inequality also accounts for the case where $\sigma=\tau$) and so (3) follows.

\end{proof}
Given a graph $G$ its profile lies in the space $\mathbb{R}^{\{0,1,\ast\}^k}$ which we will denote by $\mathbb{R}^\ast$. In view of Proposition~\ref{constraints} we define the function $F: \mathbb{R}^\ast\to\mathbb{R}$ by
$$
F(x)=\left(\sum_{\tau\in\{0,1,\ast\}^k}x_\tau\right)^2-2\sum_{\{\sigma, \tau\}\in D}x_\sigma x_\tau-\sum_{\tau\in\{0,1,\ast\}^k}\omega(\tau)x_\tau.
$$
Let us also define the following subsets of $\mathbb{R}^\ast$.
\\

\noindent
\bm{$X(\gamma)$}: For $\gamma\geq0$, let $X(\gamma)$ denote the set of elements $x\in\mathbb{R}^\ast$ satisfying:\\
\begin{enumerate}[(X1)]

\item$
F(x)\leq\gamma
$\\
\item
$
x_\tau\leq1+\gamma\hspace{0.2 cm}\mbox{whenever}\hspace{0.2 cm}w(\tau)=1.
$\\
\item $x_\tau x_\sigma\leq\gamma$ whenever $\sigma$ and $\tau$ are incompatible.\\ 
\item $x_\tau\geq0$ for all $\tau$.\\
\end{enumerate}

Now let $G$ be as in the statement of Theorem~\ref{reduced} and let $x$ be a profile of $G$. By the above proposition we have $x/n\in X(\sqrt{\delta}k2^{2k})$ whereas we also have $\norm{x}=2^{k-1}n$. We will show that for $\delta$ small, this means that $x/n$ is an element of almost maximal norm in $X(\sqrt{\delta}k2^{2k})$. We will also show that elements of large norm in $X(\sqrt{\delta}k2^{2k})$ have a very specific structure (in fact they resemble the profile of a hypercube colouring) and so this imposes a lot of structure on $x$. For now we focus our attention on the set $X(0)$ which we denote simply by $X$. Later on, we use compactness arguments to relate properties of $X$ and $X(\gamma)$ for $\gamma$ small. 

Our next goal is to classify elements of maximal $\ell_1$-norm in $X$. To describe these elements we need a definition.

\begin{definition}\label{def:dist}
 Call a set $\mathcal{A}\subseteq \{0,1,\ast\}^k$ \emph{distinguishable} if every pair of distinct elements of $\mathcal{A}$ are distinguishable and also $\omega(\tau)\geq1$ for all $\tau\in\mathcal{A}$. 
\end{definition}
The requirement that elements have weight at least $1$ is for notational convenience later in the paper. Viewing elements of $\{0,1,\ast\}^k$ as subcubes of $Q_k$, a distinguishable set is simply a collection of disjoint subcubes of $Q_k$ (of dimension at least $1$). If this collection covers the whole cube we give it a special name.

\begin{definition}
 Call a distinguishable set $\mathcal{A}\subseteq \{0,1,\ast\}^k$ a \emph{decomposition} if\\ $\bigcup_{\tau\in\mathcal{A}}Q(\tau)=\{0,1\}^k$.
\end{definition}

Let us quickly record a simple result concerning distinguishable sets which will become useful later.

\begin{lemma}\label{dist}
Let $\mathcal{D}\subset\{0,1,\ast\}^k$ be a distinguishable set. Then
$$
\sum_{\tau\in\mathcal{D}}2^{\omega(\tau)}\leq2^k,
$$
with equality if and only if $\mathcal{D}$ is a decomposition.
\end{lemma}
\begin{proof}
This is simply the observation that a distinguishable set $\mathcal{D}$ is a collection of disjoint subcubes of $Q_k$ and so the sum of their sizes $\sum_{\tau\in\mathcal{D}}2^{\omega(\tau)}$ is bounded by the size of $Q_k$. Moreover we have equality if and only if these subcubes cover all of $Q_k$ i.e. $\mathcal{D}$ is a decomposition.
\end{proof}

We define the following subset of $\mathbb{R}^{\ast}$.
\\

\noindent
$\bm{O}$: Let O denote the set of elements $x\in\mathbb{R}^{\ast}$ satisfying:\\
\begin{enumerate}[(O1)]
\item $supp(x)$ is a decomposition where $\omega(\tau)=1$ or $2$ for all $\tau\in supp(x)$.\\
\item For all $\tau\in supp(x)$, if $\omega(\tau)=1$ then $x_\tau=1$ and if $\omega(\tau)=2$ then $x_\tau=2$.\\
\end{enumerate}
It is easy to check that $O\subseteq X$. The next proposition asserts that $O$ is the set of elements of maximal $\ell_1$-norm in $X$.

\begin{proposition}\label{props2}
If $x\in X$, then 
$
\norm{x}\leq2^{k-1}
$
with equality if and only if $x\in O$.
\end{proposition}
We note that the `if' statement in the above proposition is immediate. Indeed if $x\in O$ then $supp(x)$ is a decomposition so that $\sum_{\tau\in supp(x)}2^{\omega(\tau)}=2^k$ by Lemma~\ref{dist}. Moreover $2^{\omega(\tau)}=2x_\tau$ for all $\tau\in supp(x)$ by (O2).

\begin{definition}
If $x\in X$ is such that $\norm{x}=\sup_{z\in X}\norm{z}$ then we say that $x$ is an \emph{optimal point} of $X$. 
\end{definition}
We note that since $X$ is compact, optimal points of $X$ exist. The proof of Proposition~\ref{props2} is split over the next two sections.

\section{Compressions and a Spherical Constraint}\label{sec:compress}
 In this section we make the first steps towards a proof of Proposition~\ref{props2}. Broadly speaking we apply the combinatorial technique of `shifting' or `compression' to transform the complicated non-linear constraint in the definition of $X= X(0)$ into a spherical constraint which is much more amenable to analysis. In Section \ref{sec:convex}  we apply optimisation tools to this transformed problem. We begin with a simple lemma concerning elements of $\{0,1,\ast\}^k$.
 
 \begin{lemma}\label{compat}
If $\sigma, \tau\subseteq\{0,1,\ast\}^k$ are indistinguishable and compatible and $\omega(\tau)=1$, then $Q(\tau)\subseteq Q(\sigma)$. In particular if $\omega(\sigma)=1$ also, then $\sigma=\tau$. 
\end{lemma}
\begin{proof}
Since $\sigma, \tau$ are indistinguishable and compatible we have $\abs{Q(\tau)\cap Q(\sigma)}\geq2$ by Lemma~\ref{distinguishable}. However, $\abs{Q(\tau)}=2$ and so it follows that $Q(\tau)\subseteq Q(\sigma)$. If $\omega(\sigma)=1$ also, then clearly $Q(\sigma)= Q(\tau)$ i.e. $\sigma=\tau$.
\end{proof}
 
 \begin{definition}
Let $x\in\mathbb{R}^{\ast}$. If all pairs of (not necessarily distinct) elements of $supp(x)$ are compatible, then we say that $x$ has \emph{compatible support}.
\end{definition}
Let us note that condition (X3) (with $\gamma=0$) in the definition of the set $X$ is simply the condition that elements of $X$ have compatible support. In particular, if $x\in X$ and $\tau\in\{0,1,\ast\}^k$ has weight 0, then $x_\tau=0$ since $\tau$ is not compatible with itself. 

The following lemma establishes an important property of optimal points.  For $\tau\in \{0,1,\ast\}^k$ we let $e_\tau\in \mathbb{R}^{\ast}$ denote the standard unit vector whose entries are all 0 except the entry labelled $\tau$ which is 1.
 
 \begin{lemma}\label{zero}
Let $x\in X$ be an optimal point then $F(x)=0$.
\end{lemma}
\begin{proof}
Suppose for contradiction that $F(x)<0$. Assume first that there exists $\tau\in supp(x)$ with $\omega(\tau)\geq2$. By the continuity of $F$ we may choose $\alpha>0$ small enough so that $F(x+\alpha e_\tau)<0$. Let $x'=x+\alpha e_\tau$. Since $supp(x')=supp(x)$ it is clear that $x'\in X$. However, $\|x'\|=\|x\|+\alpha>\norm{x}$ contradicting the fact that $x$ is optimal.

We may assume then that $supp(x)$ consists only of elements of weight $1$ and therefore is a distinguishable set by Lemma~\ref{compat} and the fact that $x$ has compatible support. It follows from the definition of $F$ that
$$
F(x)=\sum_{\tau\in supp(x)}(x_\tau^2-x_\tau)<0,
$$
and so $x_\tau<1$ for some $\tau\in supp(x)$. As before there exists some $\alpha>0$ sufficiently small so that $F(x+\alpha e_\tau)<0$. Let $x'=x+\alpha e_\tau$. If we pick $\alpha$ small enough so that  $x'_\tau=x_\tau+\alpha\leq1$ also, then again we have $x'\in X$ with $\|x'\|>\|x\|$ contradicting the optimality of $x$. 
\end{proof}

We now describe the transformations alluded to at the beginning of this section. They will be of great use in simplifying our analysis of optimal points of $X$.

 \begin{definition}
Let $x\in \mathbb{R}^{\ast}$ and let $\pi, \rho\in\{0,1,\ast\}^k$ be distinct. We define the \emph{$(\pi, \rho)$-compression} of $x$, denoted $x(\pi, \rho)$, as follows:
\begin{itemize}
\item If $\omega(\rho)\geq2$, or if  $\omega(\rho)\leq1$ and $x_\pi+x_\rho<1$, then let $x(\pi, \rho)$ be the vector $x'$ with coordinates: $x'_\pi=0$, $x'_\rho=x_\pi+x_\rho$ and $x'_\tau=x_\tau$ for all $\tau\in\{0,1,\ast\}^k\backslash\{\pi, \rho\}$. 
\item  If $\omega(\rho)\leq1$ and $x_\pi+x_\rho\geq1$ then let $x(\pi, \rho)$ be the vector $x'$ with coordinates: $x'_\pi=x_\pi+x_\rho-1$, $x'_\rho=1$ and $x'_\tau=x_\tau$ for all $\tau\in\{0,1,\ast\}^k\backslash\{\pi, \rho\}$. 
\end{itemize}
If $x(\pi, \rho)=x$ then we say that $x$ is \emph{$(\pi, \rho)$-compressed}.
\end{definition}
Let $x\in X$ be an optimal point, we will be interested in instances where $x(\pi, \rho)$ is also an optimal point of $X$. We observe that if $x\in \mathbb{R}^{\ast}$ and $\pi, \rho\in\{0,1,\ast\}^k$ are distinct then $\|x(\pi, \rho)\|=\|x\|$. However, if $x\in X$ then it does not follow in general that $x(\pi, \rho)\in X$.

For reasons that will become clear, we only consider $(\pi, \rho)$-compressions in the case where $\pi$ and $\rho$ are indistinguishable. It will therefore be useful to associate to each point $x\in X$, the digraph $D(x)=(V(x), E(x))$ where $V(x)=supp(x)$ and
$$E(x)=\left\{(\pi,\rho): \pi, \rho\ \text{are distinct, indistinguishable}\ \text{and}\ x(\pi, \rho)\in X\right \}.$$
In particular if $x\in X$ is $(\pi, \rho)$-compressed, where $\pi$ and $\rho$ are distinct and indistinguishable, then $(\pi, \rho)\in E(x)$. We draw attention to the fact that edges of $D(x)$ only occur between indistinguishable pairs. Conversely, the following lemma shows that, when $x\in X$ is optimal, at least one edge occurs between any indistinguishable pair in $D(x)$. 

\begin{lemma}\label{compress}
Let $x\in X$ be optimal and suppose that $\pi,\rho\in V(x)$, are indistinguishable and distinct. Then one of the following holds:
\begin{enumerate}[(i)]
\item $x$ is $(\pi, \rho)$-compressed, $x_\rho=1$, $\omega(\rho)=1$, $\omega(\pi)\geq2$ and $(\rho, \pi)\notin E(x)$,
\item $x$ is $(\rho, \pi)$-compressed, $x_\pi=1$, $\omega(\pi)=1$, $\omega(\rho)\geq2$ and $(\pi, \rho)\notin E(x)$,
\item $(\rho, \pi)$ and $(\pi, \rho)$ both lie in $E(x)$.
\end{enumerate}
\end{lemma}
\begin{proof}
Recall that 
$$
F(x)=\left(\sum_{\tau\in\{0,1,\ast\}^k}x_\tau\right)^2-2\sum_{\{\sigma, \tau\}\in D}x_\sigma x_\tau-\sum_{\tau\in\{0,1,\ast\}^k}\omega(\tau)x_\tau,
$$
where $D$ is the set of unordered distinguishable pairs from $\{0,1,\ast\}^k$. Since $\pi, \rho$ are indistinguishable, the term $x_\rho x_\pi$ does not appear in the sum $\sum_{\{\sigma, \tau\}\in D}x_\sigma x_\tau$. We may therefore express $F(x)$ in the form
\begin{equation}\label{lin}
F(x)=\left(\sum_{\tau\in\{0,1,\ast\}^k}x_\tau\right)^2-Ax_\rho-Bx_\pi-C,
\end{equation}
where $A, B$ and C do not depend on $x_\rho$ or $x_\pi$ and $A,B\geq0$. Suppose that $A\geq B$ and let $x'=x(\pi, \rho)$. Let us show that $(\pi, \rho)\in E(x)$ i.e. $x'\in X$. By \eqref{lin} we have $F(x')\leq F(x)$ and so $x'$ satisfies $(X1)$ in the definition of $X$. By the definition of $(\pi, \rho)$-compression it is clear that $x'$ also satisfies $(X2)$ and $(X4)$. Since $\rho\in supp(x)$, we also have $supp(x')\subseteq {supp}(x)$. Since $x$ has compatible support the same is true for $x'$ i.e. $x'$ satisfies (X3) and so $x'\in X$. Note that since compressions preserve the $\ell_1$-norm, $x'$ is also an optimal point of $X$.

In the case $A=B$, an identical argument shows that $(\rho, \pi)\in E(x)$ also, and so (iii) holds.

Suppose then that $A>B$. In this case, looking again at \eqref{lin}, we see that if $x$ is not $(\pi, \rho)$-compressed then we in fact have $F(x')<F(x)=0$, contradicting Lemma~\ref{zero}. We conclude that $x$ is $(\pi, \rho)$-compressed. Suppose $\omega(\rho)\geq2$, then by the definition of $(\pi, \rho)$-compression we have $x_\pi=x'_\pi=0$ contradicting the fact that $\pi\in supp(x)$ and so $\omega(\rho)=1$. Since $\pi$ and $\rho$ are compatible, indistinguishable and distinct, it follows from Lemma~\ref{compat} that $\omega(\pi)\geq2$. Let $x''=x(\rho, \pi)$. It follows that $x''_\rho=0$ and $x''_\pi=x_\rho+x_\pi$ and so by \eqref{lin}, $F(x'')>F(x)=0$. We conclude that $x''\notin X$ i.e. $(\rho, \pi)\notin E(x)$. The fact that $x_\rho=1$ follows from the fact that $\omega(\rho)=1$ and $x$ is $(\pi,\rho)$-compressed. Thus (i) holds, and similarly if $A<B$ then (ii) holds.
\end{proof}
We obtain the following immediate corollary.
\begin{corollary}\label{indep}
Let $x\in X$ be an optimal point and suppose that $I$ is an independent set in $D(x)$. Then $I$ is a distinguishable set.
\end{corollary}

\begin{definition}
We call an optimal point $x\in X$ \emph{compressed} if it is $(\pi,\rho)$-compressed for all $(\pi, \rho)\in E(x)$.
\end{definition}
 
We now show that compressed optimal points of $X$ exist. In fact we show that given any optimal point of $x\in X$ we may obtain a compressed optimal point by applying a finite number of compressions to $x$. The simpler structure of compressed optimal points will make it easier to bound their $\ell_1$-norm which is the goal of Proposition~\ref{props2}. 
 
\begin{lemma}\label{exist} 
Compressed optimal points of $X$ exist.
\end{lemma}
\begin{proof}
Let $x$ be an arbitrary optimal point of $X$ and define a sequence $x_0, x_1, x_2, \ldots$ of elements of $X$ recursively as follows: Set $x_0=x$. Having chosen $x_0,\ldots,x_t$, if $x_t$ is compressed then stop the sequence at $x_t$. If not, then there exists $(\pi, \rho)\in E(x_t)$ such that $x_t$ is not $(\pi, \rho)$-compressed. By Lemma~\ref{compress}, we must therefore have that $x(\pi, \rho)$ and $x(\rho, \pi)$ are both optimal points of $X$. Note that by the definition of $D(x_t)$, $\rho$ and $\pi$ are indistinguishable and $\{\pi, \rho\}\subseteq supp(x_t)$. Since $x_\tau$ has compatible support, it follows from Lemma~\ref{compat} that either $\omega(\pi)\geq2$ or $\omega(\rho)\geq2$. If $\omega(\rho)\geq2$ then set $x_{t+1}=x(\pi, \rho)$, if not (so that $\omega(\pi)\geq2$) set $x_{t+1}=x(\pi, \rho)$. In either case $x_{t+1}$ is an optimal point of $X$ satisfying $\abs{V(x_{t+1})}=\abs{V(x_t)}-1$. Since $0\leq\abs{V(x)}\leq3^k$ for all $x\in X$ it follows that the sequence must terminate in at most $3^k$ steps.
\end{proof}

Having discovered compressed optimal points, we now explore some of their properties. First we need a definition.\begin{definition}
A \emph{star} is a digraph with vertex set $\{\rho, \pi_1,\ldots, \pi_m\}$ (for some $m\geq0$) and edge set $\{(\rho, \pi_1), \ldots, (\rho, \pi_m)\}$. We refer to $\rho$ as the \emph{root} of the star and we call $\pi_1,\ldots, \pi_m$ \emph{leaves}. Note that we have included the possibility of a star with no leaves. 
\end{definition}

\begin{lemma}\label{star}
Let $x\in X$ be a compressed optimal point, then $D(x)$ is a disjoint union of stars. Moreover if $\rho$ is a root of positive outdegree then $\omega(\rho)\geq2$ and if $\pi$ is a leaf then $\omega(\pi)=1$ and $x_\pi=1$.
\end{lemma}
\begin{proof}
It suffices to prove the following:
\begin{enumerate}
\item If $(\rho, \pi)\in E(x)$ then $\omega(\rho)\geq2$, $\omega(\pi)=1$, $x_\pi=1$ and $(\pi, \rho)\notin E(x)$.
\item If $(\rho_1, \pi), (\rho_2, \pi)\in E(x)$ then $\rho_1=\rho_2$.
\end{enumerate}

Suppose that $(\rho, \pi)\in E(x)$, in particular $\rho$ and $\pi$ are indistinguishable. If $(\pi, \rho)\in E(x)$ also, then since $x$ is compressed we have by definition that $x(\rho,\pi)=x=x(\pi, \rho)$. However, from the definition of compression we see that the only way we can have $x(\rho,\pi)=x(\pi, \rho)$ is if $\omega(\pi)=\omega(\rho)=1$. But then by Lemma~\ref{compat}, $\pi=\rho$, a contradiction. We conclude that $(\pi, \rho)\notin E(x)$ and so (1) follows from Lemma~\ref{compress}.

Suppose now that $(\rho_1, \pi), (\rho_2, \pi)\in E(x)$. By (1) we know that $\omega(\pi)=1$ and $\omega(\rho_i)\geq2$ for $i=1,2$. By Lemma~\ref{compat} it follows that $Q(\pi)\subseteq Q(\rho_1)\cap Q(\rho_2)$ and so $\rho_1, \rho_2$ are indistinguishable by Lemma~\ref{distinguishable}. If $\rho_1\neq\rho_2$ then by Lemma~\ref{compress} we have that either $(\rho_1, \rho_2)\in E(x)$ or $(\rho_2, \rho_1)\in E(x)$, but this contradicts (1).
\end{proof}

Given a compressed optimal point $x\in X$ let
$$
L(x)=\{\tau\in V(x): d^{-}(\tau)>0\}
$$
and 
$$
R(x)=V(x)\backslash L(x).
$$
By Lemma~\ref{star}, $L(x)$ and $R(x)$ are the set of leaves and the set of roots of $D(x)$ respectively. 

\begin{lemma}\label{sphere}
Let $x\in X$ be a compressed optimal point. Then 
$$
F(x)=\sum_{\tau\in R(x)}\left(x_\tau^2+(2d^{+}(\tau)-\omega(\tau))x_\tau\right).
$$
\end{lemma}
\begin{proof}
 Lemma~\ref{star} shows that for any indistinguishable pair $\pi, \rho\in V(x)$, where $\pi\neq\rho$, exactly one of $(\pi, \rho)$ and $(\rho, \pi)$ is in $E(x)$ and so
$$
F(x)=\sum_{\tau\in V(x)}x_\tau^2+2\sum_{(\sigma, \tau)\in E(x)}x_\sigma x_\tau-\sum_{\tau\in V(x)}\omega(\tau)x_\tau.
$$
By Lemma~\ref{star} we may write
$$
\sum_{(\sigma, \tau)\in E(x)}x_\sigma x_\tau=\sum_{\rho\in R(x)}x_\rho\left(\sum_{\pi: (\rho, \pi)\in E(x)}x_\pi\right)=\sum_{\rho\in R(x)}d^{+}(\rho)x_\rho.
$$
Moreover by Lemma~\ref{star} we have $\sum_{\tau\in L(x)}(x_\tau^2-\omega(\tau)x_\tau)=0$. The result follows. 
\end{proof}

The key feature here is that for a compressed optimal point $x$, the constraint
equation $F(x) = 0$ is spherical. This allows us to more easily apply standard optimisation techniques and this will be the concern of the next section. For now it will be useful for us to establish some degree conditions on the vertices of $D(x)$ for a compressed optimal point $x\in X$.

\begin{lemma}\label{degree0}
Let $x\in X$ be a compressed optimal point, then $d^{+}(\sigma)\leq2^{\omega(\sigma)-1}$ for all $\sigma\in V(x)$.
\end{lemma}
\begin{proof}
Suppose that $\rho\in V(x)$ is such that $d^{+}(\rho)>0$. By Lemma~\ref{star}, $\rho$ is the root of a star in $D(x)$. Let $L$ be the set of leaves of this star (so in particular $\abs{L}=d^{+}(\rho)$ and $\omega(\pi)=1$ for all $\pi\in L$). Note that $L$ is an independent set in $D(x)$ and therefore it is a distinguishable set by Corollary~\ref{indep}. Note further that for each $\pi\in L$, $\rho$ and $\pi$ are indistinguishable and compatible and hence $Q(\pi)\subseteq Q(\rho)$ by Lemma~\ref{compat}. It follows that
\[
2d^{+}(\rho)=\sum_{\pi\in L}\abs{Q(\pi)}\leq\abs{Q(\rho)}=2^{\omega(\rho)}.
\]
\end{proof}

We can now bootstrap, using the previous two lemmas to establish a much stronger degree condition. The idea behind the proof of the following lemma is readily explained however it is notationally laborious. The idea is that if $x\in X$ is a compressed optimal point and a star in $D(x)$ with root $\rho$ has $>\omega(\rho)$ leaves, then one can contradict the optimality of $x$ by replacing this star with a collection of stars whose roots have less weight. First let us generalise an earlier notation.

\begin{definition}
Let $\sigma\in\{0,1,\ast\}^k$ and let $W=\{i\in[k]: \sigma_i=\ast\}$. Then for $S\subseteq W$ define
\[
Q(\sigma; S)=\{\tau\in\{0,1,\ast\}^k: \tau_i\in\{0,1\}\ \text{for}\ i\in W\backslash S\ \text{and}\ \tau_i=\sigma_i\ \text{otherwise}\}.
\]
\end{definition}
Note that elements of $Q(\sigma; S)$ are pairwise distinguishable and that $Q(\sigma; \emptyset)$ is simply the set $Q(\sigma)$. We may think of $Q(\sigma; S)$ as a decomposition of $Q(\sigma)$ into `parallel' subcubes of dimension $\abs{S}$.

\begin{lemma}\label{degree}
Let $x\in X$ be a compressed optimal point, then $d^{+}(\sigma)\leq\omega(\sigma)$ for all $\sigma\in V(x)$.
\end{lemma}
\begin{proof}
By Lemma~\ref{star} we may write $V(x)=S_1\cup\ldots\cup S_q$, a disjoint union where each $S_i$ is the vertex set of a star in $D(x)$. Suppose that there exists $\sigma\in V(x)$ such that $d^{+}(\sigma)>\omega(\sigma)$. Without loss of generality assume $\sigma$ is the root of  $S_1$. By Lemma~\ref{sphere} we then have that $\omega(\sigma)< d^{+}(\sigma)\leq2^{\omega(\sigma)-1}$ and so $\omega(\sigma)\geq3$. Without loss of generality assume that $\sigma_1=\sigma_2=\ast$. By Lemma~\ref{star} we have $x_\tau=1$ for all $\tau\in L(x)$ and so 
\begin{equation}\label{abs1}
\|x\|=\abs{L(x)}+\sum_{\tau\in R(x)}x_\tau.
\end{equation}
 We proceed by modifying $x$, being careful to stay within the set $X$. Take $\pi\in Q(\sigma; \{1,2\})$ and note that $\omega(\pi)=2$. Consider now the element $x'\in \mathbb{R}^{\ast}$ defined as follows. Let $x'_\pi=x_\sigma$, $x'_\tau=1$ for all $\tau\in Q(\sigma; \{1\})$, $x'_\tau=x_\tau$ for all $\tau\in S_2\cup\ldots\cup S_q$ and $x'_\tau=0$ otherwise. We now check that $x'\in X$. Clearly $x'_\tau\leq1$ whenever $\omega(\tau)=1$. Note that $supp(x')=\{\pi\}\cup Q(\sigma; \{1\})\cup S_2\cup\ldots\cup S_q$. Now, if $\tau\in \{\pi\}\cup Q(\sigma; \{1\})$ we have $Q(\tau)\subseteq Q(\sigma)$ and since $\sigma\in S_1$, we know that $\sigma$, and hence also $\tau$, is distinguishable from each element of $S_2\cup\ldots\cup S_k$. Since $\tau_1=\ast$ for each $\tau\in \{\pi\}\cup Q(\sigma; \{1\})$ we see that $\{\pi\}\cup Q(\sigma; \{1\})$ contains no incompatible pairs. It follows that $x'$ has compatible support. Finally, note that by a calculation similar to that in the proof of Lemma~\ref{sphere} we have
\begin{align}\label{array}
F(x')&=x_\sigma^2+2x_\sigma+ \sum_{\tau\in R(x)\backslash \{\sigma\}}\left(x_\tau^2+\left(2d^{+}(\tau)-\omega(\tau)\right)x_\tau\right)\\
&=F(x)-(2d^{+}(\sigma)-\omega(\sigma)-2)x_\sigma.
\end{align}
Recalling that $d^{+}(\sigma)>\omega(\sigma)$ we have $2d^{+}(\sigma)-\omega(\sigma)>\omega(\sigma)\geq3$. Since $\sigma\in supp(x)$, we also have $x_\sigma>0$ and so \eqref{array} implies that $F(x')<F(x)=0$. Thus we do indeed have $x'\in X$. Note that 
\begin{equation}\label{abs2}
\|x'\|=\abs{L(x)}-d^{+}(\sigma)+\abs{Q(\sigma; \{1\})} + \sum_{\tau\in R(x)}x_\tau,
\end{equation}
and observe that $d^{+}(\sigma)\leq2^{\omega(\sigma)-1}=\abs{Q(\sigma; \{1\})}$ by Lemma~\ref{sphere}. It now follows from \eqref{abs1} and \eqref{abs2} that $\|x'\|\geq\|x\|$, and so $x'$ is an optimal point of $X$. However we have shown that $F(x')<0$ contradicting Lemma~\ref{zero}.
\end{proof}
Gathering all the information we have obtained on compressed optimal points, we show that a proof of the following proposition is almost enough to deduce Proposition~\ref{props2}. Let us remind ourselves that in the definition of a distinguishable set (Definition \ref{def:dist}), we require all elements of the set to have weight at least $1$.

\begin{proposition}\label{convex}
Let $\mathcal{D}\subseteq\{0,1,\ast\}^k$ be a distinguishable set and let $\Delta=\{d_\tau : \tau\in \mathcal{D}\}$ be a set of integers satisfying $0\leq d_\tau\leq\omega(\tau)$ for all $\tau\in \mathcal{D}$, and $d_\tau=0$ whenever $\omega(\tau)=1$. Suppose that $x\in\mathbb{R}^{\ast}$ is a vector with $supp(x)=\mathcal{D}$ satisfying 
\begin{enumerate}
\item $$\sum_{\tau\in \mathcal{D}}\left(x_\tau^2+(2d_\tau-\omega(\tau))x_\tau\right)=0,$$
\item $x_\tau\leq1$ whenever $\omega(\tau)=1$.
\end{enumerate}
Then 
$$
\sum_{\tau\in \mathcal{D}}x_\tau\leq 2^{k-1}-\sum_{\tau\in \mathcal{D}}d_\tau.
$$
Furthermore we have equality only if $x\in O$ and $\Delta=\{0\}$.
\end{proposition}
A proof of Proposition~\ref{convex} will be the focus of the next section, for now we note that it has the following corollary

\begin{corollary}\label{Ocompress}
$O$ is the set of compressed optimal points of $X$. In particular, if $x\in X$ then $\norm{x}\leq2^{k-1}$.
\end{corollary}
\begin{proof}
Let $x\in X$ be a compressed optimal point. Note that $R(x)$ is an independent set in the digraph $D(x)$ and hence by Corollary~\ref{indep}, $R(x)$ is a distinguishable set. Set $\Delta=\{d^{+}(\tau): \tau\in R(x)\}$ and let $x'$ be the element of $\mathbb{R}^{\ast}$ supported on $R(x)$ such that $x'_\tau=x_\tau$ for all $\tau\in R(x)$. Note that by Lemmas~\ref{zero}, \ref{sphere}, \ref{degree} and by the definition of the set $X$, we have that $\Delta$ and $x'$ satisfy the conditions in the statement of Proposition~\ref{convex}. Assuming Proposition~\ref{convex}, it therefore follows that 
\begin{equation}\label{latter}
\sum_{\tau\in R(x)}x_\tau\leq 2^{k-1}-\sum_{\tau\in R(x)}d^{+}(\tau),
\end{equation}
with equality only if $x'\in O$ and $d^{+}(\tau)=0$ for all $\tau\in R(x)$. The latter condition implies that $x'=x$ and so we have equality in \eqref{latter} only if $x\in O$. By Lemma~\ref{star},
$$
\|x\|=\sum_{\tau\in R(x)}x_\tau+\sum_{\tau\in R(x)}d^{+}(\tau),
$$
and so it follows that $\norm{x}\leq 2^{k-1}$ with equality only if $x\in O$. The result follows by noting that for all $z\in O$, $\norm{z}=2^{k-1}$ and $z$ is compressed.

\end{proof}

It is now clear that Proposition~\ref{props2} would follow if we could also prove the following.

\begin{proposition}\label{opt}
If $x\in X$ is an optimal point, then $x$ is compressed.
\end{proposition}

We prove Proposition~\ref{opt} in Section \ref{sec:stab}.

\section{Constrained Optimisation and a proof of Proposition~\ref{convex}}\label{sec:convex}
In this section we prove Proposition~\ref{convex} thus finalising the main stepping stone toward a proof of Proposition~\ref{props2}. We use standard tools from the theory of convex optimisation to exploit the spherical constraint found in the previous section. This will lead us to consider the possible distributions of weights in distinguishable sets which we optimise over in a separate argument. The main tools that we borrow are the Karush-Kuhn-Tucker (KKT) conditions along with Slater's constraint qualification. Below is a statement of the result we use, phrased to match our needs (see~\cite[p. 244]{KKT} for a detailed account).

\begin{theorem}[KKT + Slater's Condition]\label{KKT}
Let $f, g_1, \ldots, g_r: \mathbb{R}^m\to\mathbb{R}$ be convex, differentiable functions and let 
\[
S=\{x\in\mathbb{R}^m: g_i(x)\leq0 \ \text{for}\ i=1,\ldots, r\}.
\]
Suppose that there exists an $x_0\in\mathbb{R}^m$ such that $g_i(x_0)<0$ for $i=1,\ldots,r$. Then if $x^*\in S$ is such that
\[
f(x^*)=\sup_{x\in S}f(x),
\] 
then there exist $\lambda_1,\ldots,\lambda_r\in\mathbb{R}$ such that
\begin{enumerate}[(i)]
\item $\nabla f(x^*)=\sum_{i=1}^r\lambda_i\nabla g_i(x^*)$,
\item $\lambda_i\geq0$,  $\ i=1,\ldots,r$,
\item $\lambda_i g_i(x^*)=0$,  $\ i=1,\ldots,r$.
\end{enumerate}
\end{theorem}

In view of the statement of Proposition~\ref{convex} it is natural to apply Theorem~\ref{KKT} to establish the following.

\begin{lemma}\label{lag}
Let $\alpha_1, \ldots, \alpha_m$ be integers where $\alpha_i=1$ for $i=1,\ldots,\ell$, ($0\leq\ell\leq m$) and let $x_1,\ldots,x_m$ be real numbers satisfying
\begin{enumerate}[(i)]
\item $
\sum_{i=1}^m(x_i^2-\alpha_ix_i)\leq0,
$
\item $x_i\leq1$ for $i=1,\ldots,\ell$.
\end{enumerate}
If $\sum_{i=1}^m \alpha_i^2>m$, then
$$
\sum_{i=1}^mx_i\leq\ell+\frac{1}{2}\left(\sum_{i=\ell+1}^m\alpha_i+\sqrt{(m-\ell)\sum_{i=\ell+1}^m\alpha_i^2}\right),
$$
with equality only if $x_i=1$ for $i\leq\ell$ and
$
x_i=\frac{1}{2}\left(\alpha_i+\sqrt{\frac{1}{(m-\ell)}\sum_{i=\ell+1}^m\alpha_i^2}\right)\ \text{for}\ i>\ell.
$\\

\noindent If instead $\sum_{i=1}^m \alpha_i^2\leq m$, then
$$
\sum_{i=1}^mx_i\leq\frac{1}{2}\left(\sum_{i=1}^m\alpha_i+\sqrt{m\sum_{i=1}^m\alpha_i^2}\right),
$$
with equality only if
$
x_i=\frac{1}{2}\left(\alpha_i+\sqrt{\frac{1}{m}\sum_{i=1}^m\alpha_i^2}\right)\ \text{for all}\ i.
$
\end{lemma}
\begin{proof}
Note first that if $\alpha_i=0$ for all $i$ (so in particular $\ell=0$) then inequality $(i)$ implies that $x_i=0$ for all $i$ in which case there's nothing to prove. Suppose then that this is not the case and define functions $f, g_1, \ldots, g_{\ell+1}: \mathbb{R}^m\to\mathbb{R}$ as follows. Let $f(x)=\sum_{i=1}^mx_i$, $g_i(x)=x_i-1$ for $i=1,\ldots,\ell$ and $g_{\ell+1}(x)=\sum_{i=1}^m(x_i^2-\alpha_ix_i)$. Note that the functions just defined are all convex and differentiable. Let $x_0=(\alpha_1/2,\ldots,\alpha_m/2)$, the centre of the spherical region described by (ii) and observe that $g_i(x_0)<0$ for $i=1,\ldots,\ell+1$.  Let $S=\{x\in\mathbb{R}^m: g_i(x)\leq0 \ \text{for}\ i=1,\ldots, \ell+1\}$. $S$ is compact and $f$ is continuous hence we may pick $x^*\in S$ such that 
$$
f(x^*)=\sup_{x\in S}f(x).
$$
Let $x^*=(x_1,\ldots,x_m)$. By Theorem~\ref{KKT}, there exist real numbers $\lambda_1,\ldots,\lambda_\ell$ and $\Lambda$ such that the following hold (for notational convenience we define $\lambda_j=0$ for $j>\ell$):
\begin{enumerate}
\item $\Lambda(2x_i-\alpha_i)+\lambda_i=1$ for all $i$,
\item $\Lambda\geq0$ and $\lambda_i\geq0$ for all $i$,
\item $\Lambda\left(\sum_{i=1}^m(x_i^2-\alpha_ix_i)\right)=0$ and $\lambda_i(x_i-1)=0$ for all $i$.
\end{enumerate}
We consider three cases depending on the value of $\Lambda$. First let us suppose that $\Lambda=0$. In this case, by (1) we must have $\lambda_i=1$ for all $i$. Recalling that $\lambda_j=0$ for $j>\ell$ by definition, we must also have $\ell=m$ and so $\alpha_i=1$ for all $i$. Moreover, it follows from (3) that $x_i=1$ for all $i$ and so we're done.

By (2), we may now assume that $\Lambda>0$ and so we may rewrite (1) as
\begin{equation}\label{xi}
x_i=\frac{1}{2}\left(\frac{1-\lambda_i}{\Lambda}+\alpha_i\right)\ \mbox{for all }i.
\end{equation}
Moreover, $\sum_{i=1}^m(x_i^2-\alpha_ix_i)=0$ by (3) which by \eqref{xi} gives
\begin{equation}\label{alpsq}
\frac{1}{\Lambda^2}\sum_{i=1}^m(1-\lambda_i)^2=\sum_{i=1}^m\alpha_i^2.
\end{equation}
Now, note that for $i\leq\ell$ we have $\alpha_i=1$ and $x_i\leq1$ and so by \eqref{xi} we have 
\begin{equation}\label{mu}
1-\Lambda\leq\lambda_i\ \mbox{for}\ i\leq\ell.
\end{equation}
If $\Lambda<1$ then by \eqref{mu} we have $\lambda_i>0$ for $i\leq\ell$ and so by (3), $x_i=1$ for $i\leq\ell$ and so in fact by \eqref{xi}
\begin{equation}\label{eq}
1-\Lambda=\lambda_i\ \mbox{for}\ i\leq\ell.
\end{equation}
Recalling that $\alpha_i=1$ for $i\leq\ell$ and $\lambda_i=0$ for $i>\ell$ by definition, \eqref{alpsq} then gives
\begin{equation}\label{eq:rms}
\frac{1}{\Lambda}=\sqrt{\frac{1}{m-\ell}\sum_{i=\ell+1}^m\alpha_i^2}.
\end{equation}
From \eqref{xi} it now follows that
\[
x_i=\frac{1}{2}\left(\alpha_i+\sqrt{\frac{1}{m-\ell}\sum_{i=\ell+1}^m\alpha_i^2}\right)\ \text{for}\ i>\ell,
\]
and so
\[
\sum_{i=1}^mx_i=\ell+\frac{1}{2}\left(\sum_{i=\ell+1}^m\alpha_i+\sqrt{(m-\ell)\sum_{i=\ell+1}^m\alpha_i^2}\right).
\]
Recalling that $\Lambda<1$, it follows from \eqref{eq:rms} that $\sum_{i=1}^m\alpha_i^2>m$.

It remains to consider the case where $\Lambda\geq1$. Recall that if $\lambda_i>0$ for some $i$ then $x_i=1$ by (3) and so $\Lambda=1-\lambda_i$ by \eqref{xi}. However this contradicts the assumption that $\Lambda\geq1$ and so we conclude that $\lambda_i=0$ for all $i$. It follows from \eqref{alpsq} that 
\begin{equation}\label{eq:rms2}
\frac{1}{\Lambda}=\sqrt{\frac{1}{m}\sum_{i=1}^m\alpha_i^2},
\end{equation}
so that by \eqref{xi},
\[
x_i=\frac{1}{2}\left(\alpha_i+\sqrt{\frac{1}{m}\sum_{i=1}^m\alpha_i^2}\right)\ \text{for all}\ i.
\]
The result follows, noting that by \eqref{eq:rms2} we have $\sum_{i=1}^m\alpha_i^2\leq m$ in this case. 
\end{proof}

We are almost ready to prove Proposition~\ref{convex}, but first we need the following inequality.

\begin{lemma}\label{shift}
Let $\alpha_{1},\ldots,\alpha_m$ be integers $\geq$ 2 then
$$
\sum_{i=1}^m\alpha_i+\sqrt{m\sum_{i=1}^m\alpha_i^2}\leq \sum_{i=1}^m2^{\alpha_i},
$$
and equality holds if only if $\alpha_i=2$ for all $i$.
\begin{proof}
Let $\alpha=(\alpha_1,\ldots, \alpha_m)$. We induct on the value of $S_\alpha:=\sum_{i=1}^m(2^{\alpha_i-2}-1)$. If $S_\alpha=0$, then $\alpha_i=2$ for all $i$, so that 
$$
\sum_{i=1}^m\alpha_i+\sqrt{m\sum_{i=1}^m\alpha_i^2}=4m=\sum_{i=1}^m2^{\alpha_i}.
$$
Suppose then that $S_\alpha>0$ so that $\alpha_j\geq3$ for some $j\in[m]$. Without loss of generality assume that $j=1$. Define a new sequence of integers $\alpha'=(\alpha_1',\ldots,\alpha_{m+1}')$, as follows: Let $\alpha_1'=\alpha_2'=\alpha_1-1$ and $\alpha_i'=\alpha_{i-1}$ for $i=3,4,\ldots,m+1$. Note that $\alpha_i'\geq2$ for all $i$ and $S_{\alpha'}=S_\alpha-1$ and so by the inductive hypothesis
\begin{equation}\label{ind}
\sum_{i=1}^{m+1}\alpha'_i+\sqrt{(m+1)\sum_{i=1}^{m+1}\alpha_i'^2}\leq \sum_{i=1}^{m+1}2^{\alpha'_i},
\end{equation}
  Note that 
\begin{equation}\label{1}
\sum_{i=1}^{m+1}2^{\alpha_i'}=\sum_{i=1}^m2^{\alpha_i}\ \text{and}\ \sum_{i=1}^{m+1}\alpha_i'-\sum_{i=1}^m\alpha_i=\alpha_1-2>0,
\end{equation}
and also
\begin{equation}\label{3}
(m+1)\sum_{i=1}^{m+1}\alpha_i'^2-m\sum_{i=1}^m\alpha_i^2=\sum_{i=1}^m\alpha_i^2+(m+1)(\alpha_1^2-4\alpha_1+2)\geq \sum_{i=1}^m\alpha_i^2-(m+1),
\end{equation}
where in the last inequality we used the fact that $\alpha^2-4\alpha+2\geq-1$ for $\alpha\geq3$. Note that since $\alpha_i\geq2$ for all $i$, we certainly have that $\sum_{i=1}^m\alpha_i^2> m+1$. It follows then from \eqref{3} that 
\begin{equation}\label{4}
(m+1)\sum_{i=1}^{m+1}\alpha_i'^2> m\sum_{i=1}^m\alpha_i^2.
\end{equation}
Combining \eqref{ind}, \eqref{1}, and \eqref{4} we have
$$
\sum_{i=1}^m\alpha_i+\sqrt{m\sum_{i=1}^m\alpha_i^2}<\sum_{i=1}^{m+1}\alpha'_i+\sqrt{(m+1)\sum_{i=1}^{m+1}\alpha_i'^2}\leq \sum_{i=1}^{m+1}2^{\alpha'_i}=\sum_{i=1}^m2^{\alpha_i}
$$
as required. Note the strict inequality, and so we only have equality in the case where $\alpha_i=2$ for all $i$.
\end{proof}
\end{lemma}

\begin{proof}[Proof of Proposition~\ref{convex}]
Consider first the case where $\sum_{\tau\in\mathcal{D}}(\omega(\tau)-2d_\tau)^2>\abs{\mathcal{D}}$. Suppose that $\ell$ elements of $\mathcal{D}$ have weight 1 and let $\mathcal{D'}=\{\tau\in\mathcal{D}: \omega(\tau)\geq2\}$. Note that by Lemma~\ref{dist} we have $\sum_{\tau\in\mathcal{D}}2^{\omega(\tau)}\leq2^k$ and hence 
\begin{equation}\label{newdist}
\sum_{\tau\in\mathcal{D}'}2^{\omega(\tau)}\leq2^k-2\ell.
\end{equation}
Applying Lemma~\ref{lag}, recalling that $0\leq d_\tau\leq w(\tau)$ for all $\tau\in\mathcal{D}$, that $d_\tau=0$ whenever $\omega(\tau)=1$ and using \eqref{newdist} and Lemma~\ref{shift} we have
\begin{align}
\sum_{\tau\in\mathcal{D}}x_\tau&\leq\ell+\frac{1}{2}\left(\sum_{\tau\in\mathcal{D}'}(\omega(\tau)-2d_\tau)+\sqrt{\abs{\mathcal{D}'}\sum_{\tau\in\mathcal{D}'}(\omega(\tau)-2d_\tau)^2}\right)\label{b1}\\
&\leq\ell+\frac{1}{2}\left(\sum_{\tau\in\mathcal{D}'}\omega(\tau)+\sqrt{\abs{\mathcal{D}'}\sum_{\tau\in\mathcal{D}'}\omega(\tau)^2}\right)-\sum_{\tau\in\mathcal{D}}d_\tau\label{b2}\\
&\leq\ell+\sum_{\tau\in\mathcal{D}'}2^{\omega(\tau)-1}-\sum_{\tau\in\mathcal{D}}d_\tau\label{b3}\\
&\leq2^{k-1}-\sum_{\tau\in\mathcal{D}}d_\tau.\label{b4}
\end{align}
We analyse the conditions for equality to hold. For equality to hold in \eqref{b2} it must be the case that for all $\tau\in \mathcal{D}$, either $d_\tau=0$ or $d_\tau=\omega(\tau)$. By Lemma~\ref{shift}, for equality to hold in \eqref{b3} it must be the case that $\omega(\tau)=2$ for all $\tau\in\mathcal{D}'$. It now follows from Lemma~\ref{lag} that for equality to also hold in $\eqref{b1}$, we must have $x_\tau=1$ whenever $\omega(\tau)=1$, $x_\tau=2$ for all $\tau\in\mathcal{D}'$ such that $d_\tau=0$ and $x_\tau=0$ for all $\tau\in\mathcal{D}'$ such that $d_\tau=\omega(\tau)$. However, since each $x_\tau$ is non-zero by assumption we conclude that $d_\tau=0$ for all $\tau\in \mathcal{D}$ i.e. $\Delta=\{0\}$. Finally, for equality to hold in \eqref{b4} we must have equality in   \eqref{newdist} and so $\mathcal{D}$ is a decomposition by Lemma~\ref{dist}. It follows that $x\in O$.

It remains to consider the case where $\sum_{\tau\in\mathcal{D}}(\omega(\tau)-2d_\tau)^2\leq\abs{\mathcal{D}}$. By Lemma~\ref{lag} and Lemma~\ref{dist} we then have 
\begin{align}
\sum_{\tau\in\mathcal{D}}x_\tau&\leq\frac{1}{2}\left(\abs{\mathcal{D}}+\sum_{\tau\in\mathcal{D}}\omega(\tau)\right)-\sum_{\tau\in\mathcal{D}}d_\tau\label{a1}\\
&\leq\frac{1}{2}\left(\abs{\mathcal{D}}+\sum_{\tau\in\mathcal{D}}2^{\omega(\tau)-1}\right)-\sum_{\tau\in\mathcal{D}}d_\tau\label{a2}\\
&\leq2^{k-1}-\sum_{\tau\in\mathcal{D}}d_\tau.\label{a3}
\end{align}
For equality to hold in \eqref{a3}, we must have that $\abs{\mathcal{D}}=2^{k-1}$ and so $\mathcal{D}$ is a decomposition consisting only of elements of weight 1. It follows that $d_\tau=0$ and $x_\tau\leq1$ for all $\tau\in\mathcal{D}$. If equality holds throughout the above, we then have that $x_\tau=1$ for all $\tau\in\mathcal{D}$ and so $x\in O$.

\end{proof}

\section{Analytic and Combinatorial Stability}\label{sec:stab}
In this section we prove Proposition~\ref{opt} thus concluding our proof of Proposition~\ref{props2}. Note that Proposition~\ref{props2} classifies the optimal points of $X$. We use compactness arguments to prove a result to the effect that `almost optimal' points of $X$ must be close (in $\ell_1$ norm) to a genuine optimal point of $X$. Furthermore, compactness allows us to derive similar properties for $X(\gamma)$ when $\gamma$ is small. We then investigate what implications this has in our original combinatorial setting and complete the proof of Theorem~\ref{reduced}.

\begin{proof}[Proof of Proposition~\ref{opt}]

Let $H$ be the matrix with rows and columns indexed by $\{0,1,\ast\}^k$ where
 \[H_{\sigma \tau} 
    =
  \begin{cases}1 & \text{if $\sigma, \tau$ are indistinguishable,}\\
  0 & \text{if $\sigma, \tau$ are distinguishable.}
\end{cases}
 \]
Note that in particular, all diagonal entries of $H$ are equal to 1. Let $w=(-\omega(\tau): \tau\in\{0,1,\ast\}^k)\in\mathbb{R}^{\ast}$, then for $x\in\mathbb{R}^{\ast}$ we may write
\begin{equation}\label{altF}
F(x)=w^{T}x+x^THx.
\end{equation}

Suppose now that $x\in X$ is an optimal point. By the proof of Lemma~\ref{exist}, there is a finite sequence $x=x_0, x_1, \ldots, x_m$ of distinct optimal points of $X$ where $x_m$ is compressed, and for $i=0,\ldots, m-1$, $x_{i+1}=x_i(\pi_i, \rho_i)$ for some indistinguishable pair $\pi_i, \rho_i\in supp(x_i)$. Moreover we know that $\omega(\rho_i)\geq2$ and that $\pi_i\notin supp(x_{i+1})$ for all $i$. 

Suppose that $x$ is not compressed so that $m\geq1$. Let $y=x_{m-1}$, $z=x_m$ and let $\pi=\pi_{m-1}$, $\rho=\rho_{m-1}$. Since $z=y(\pi, \rho)$, it follows from the definition of compression that $z=y+ \alpha(e_\rho-e_\pi)$ for some $\alpha>0$. Let $p=e_\pi-e_\rho$. It follows, by the Taylor expansion of $F$, that
\begin{equation}\label{taylor}
F(y)=F(z+\alpha p)=F(z)+\alpha p^T\nabla F(z)+\alpha^2p^THp.
\end{equation}
Recall that $F(y)=F(z)=0$ by Lemma~\ref{zero}. Furthermore by direct calculation we also have $p^THp=0$. It follows from \eqref{taylor} that $p^T\nabla F(z)=0$ i.e.
\begin{equation}\label{equal}
\frac{\partial F}{\partial x_{\rlap{$\scriptstyle \pi$}}}(z)=\frac{\partial F}{\partial x_{\rlap{$\scriptstyle \rho$}}}(z).
\end{equation}
Let $I_\rho$ be the set of elements of $\{0,1,\ast\}^k$ that are indistinguishable from $\rho$ excluding $\rho$ itself. Define $I_\pi$ similarly. From the definition of $F$ we have 
\begin{equation}\label{partial}
\frac{\partial F}{\partial x_{\rlap{$\scriptstyle \rho$}}}(z)=2z_\rho+2\sum_{\tau\in I_\rho}z_\tau - \omega(\rho).
\end{equation}
Since $z$ is a compressed optimal point we have $z\in O$ by Corollary~\ref{Ocompress}. Since $\omega(\rho)\geq2$ and $\rho\in supp(z)$ we conclude that in fact $\omega(\rho)=2$ and so $z_\rho=2$. Moreover since $supp(z)$ is a distinguishable set we conclude that $z_\tau=0$ for all $\tau\in I_\rho$. It follows from \eqref{partial} that $\frac{\partial F}{\partial x_{\rlap{$\scriptstyle \rho$}}}(z)=2$ and hence from \eqref{equal} that
\begin{equation}\label{partial2}
\frac{\partial F}{\partial x_{\rlap{$\scriptstyle \pi$}}}(z)=2z_\pi+2\sum_{\tau\in I_\pi}z_\tau - \omega(\pi)=2.
\end{equation}
Since $z\in O$ we know that for all $\tau\in supp(z)$, $\omega(\tau)=1$ or $2$ and $z_\tau=\omega(\tau)$ . Let $w_1, w_2$ be the number of elements of $I_\pi\cap supp(z)$ with weights $1$, $2$ respectively. Since $\pi\notin supp(z)$, we can then infer from \eqref{partial2} that
\begin{equation}\label{w}
2w_1+4w_2-\omega(\pi)=2.
\end{equation}
We also know that $supp(z)$ is a decomposition and so
\begin{equation}\label{subset}
Q(\pi)=\bigcup_{\tau\in supp(z)}\left(Q(\tau)\cap Q(\pi)\right)=\bigcup_{\tau\in I_\pi\cap supp(z)}\left(Q(\tau)\cap Q(\pi)\right)\subseteq\bigcup_{\tau\in I_\pi\cap supp(z)}Q(\tau).
\end{equation}
The second equality comes from the fact that $Q(\tau)\cap Q(\pi)=\emptyset$ whenever $\tau$ and $\pi$ are distinguishable. Comparing the cardinality of the sets in \eqref{subset} yields
\begin{equation}\label{sumset}
2^{\omega(\pi)}\leq\ \sum_{\tau\in I_\pi\cap supp(z)}2^{\omega(\tau)}=2w_1+4w_2.
\end{equation}
Note also that $\rho\in I_\pi\cap supp(z)$ and $\omega(\rho)=2$ so that $w_2\geq1$. Using \eqref{w}, this last observation implies that $\omega(\pi)\geq2$ whereas combining \eqref{w} and \eqref{sumset} we have
\begin{equation}\label{through}
2^{\omega(\pi)}-\omega(\pi)\leq2w_1+4w_2-\omega(\pi)=2
\end{equation}
We deduce that $\omega(\pi)=2$ and so we have equality throughout \eqref{through}, in particular we have equality in \eqref{sumset} and so also in \eqref{subset}. Note that $\abs{Q(\pi)}=\abs{Q(\rho)}$ since $\omega(\pi)=\omega(\rho)=2$. Recalling that $\rho\in I_\pi\cap supp(z)$ equality in \eqref{subset} would therefore imply that $Q(\pi)=Q(\rho)$ i.e. $\pi=\rho$. This is a contradiction and so $x$ must be compressed. \end{proof}

Proposition~\ref{props2} has the following corollary that says an almost optimal point of $X$ must be close in norm to an actual optimal point of $X$.
\begin{proposition}\label{close2}
Let $\eta\ll\eps$. If $x\in X$ satisfies $\norm{x}>2^{k-1}-\eta$, then there exists an $x^\ast\in O$ such that
$
\norm{x-x^\ast}<\eps.
$
\end{proposition}
\begin{proof}
Consider the set
\[
\tilde{X}:=X\big\backslash\bigcup_{x^\ast\in O}B_\eps(x^\ast).
\]
$\tilde{X}$ is compact and so $\sup_{z\in\tilde{X}}\norm{z}=\norm{\tilde{x}}$ for some $\tilde{x}\in\tilde{X}$. By the definition of $\tilde{X}$, $\tilde{x}\notin O$ and so by Proposition~\ref{props2}, $\norm{\tilde{x}}=2^{k-1}-\eta$ for some $\eta>0$. It follows that if $x\in X$ satisfies $\norm{x}>2^{k-1}-\eta$ then $x\notin\tilde{X}$ and so $x\in B_\eps(x^\ast)$ for some $x^\ast\in O$.
\end{proof}

The following lemma allows us to relate properties of $X$ and $X(\gamma)$ for $\gamma$ small.

\begin{lemma}\label{compact}
Let $\gamma\ll\eta$. If $x\in X(\gamma)$, then there exists $x_0\in X$ for which $\norm{x-x_0}<\eta$.
\end{lemma}
\begin{proof}
Let $(\gamma_i)_{i\in\mathbb{N}}$ be a strictly decreasing sequence tending to 0, and let $X_i=X(\gamma_i)$ for $i\in \mathbb{N}$. Then $X_1, X_2, \ldots$ is a decreasing sequence of compact sets i.e. $X_{i+1}\subseteq X_i$ for $i\in \mathbb{N}$. Consider the set
$$
U=\bigcup_{z\in X}B_\eta(z),
$$
an open set containing $X$. Note that $(X_i\backslash U)_{i\in\mathbb{N}}$ is also a decreasing sequence of compact sets and also
$$
\bigcap_{i=1}^{\infty}(X_i\backslash U)=\left(\bigcap_{i=1}^{\infty}X_i\right)\Big\backslash U=X\backslash U=\emptyset.
$$
By Cantor's Intersection Theorem (see~\cite[Theorem 2.36, p.38]{rudin}) it follows that $X_m\backslash U=\emptyset$ for some $m\in\mathbb{N}$. In other words, if $x\in X(\gamma_m)$ then $x\in U$ so that $x\in B_\eta(x_0)$ for some $x_0\in X$. The result follows by taking $\gamma\leq\gamma_m$. 
\end{proof}

\begin{corollary}\label{close1}
Let $\gamma\ll\eps$. If $x\in X(\gamma)$ satisfies $\|x\|= 2^{k-1}$, then there exists an $x^\ast\in O$ such that $\|x-x^\ast\|<\varepsilon$.
\end{corollary}
\begin{proof}
Given $\eps>0$, let $\eta=\min\{\eta_{\ref{close2}}(\eps/2), \eps/2\}$ and suppose that $\gamma\leq\gamma_{\ref{compact}}(\eta)$. Suppose that $x\in X(\gamma)$ satisfies $\norm{x}=2^{k-1}$. By Lemma~\ref{compact} there exists an $x_0\in X$ such that $\norm{x_0-x}<\eta$ and so $\norm{x_0}>\norm{x}-\eta=2^{k-1}-\eta$. It follows from Proposition~\ref{close2} that there exists an $x^\ast\in O$ such that $\norm{x_0-x^\ast}<\eps/2$ and so 
$$
\norm{x-x^\ast}\leq\norm{x-x_0}+\norm{x_0-x^\ast}<\eta+\eps/2\leq \eps.
$$
\end{proof}
Let
$$
O^\ast=\{x\in O: \omega(\tau)=1\ \text{for all}\ \tau\in supp(x)\}.
$$
In words, $O^\ast$ is the set of all elements $x\in \mathbb{R}^{\ast}$ such that $x$ is supported on a perfect matching of $Q_k$ and all non-zero entries of $x$ are equal to 1. We can also view $O^{\ast}$ as the set of profiles of hypercube colourings normalised by clique size. Our aim is to use the stability-type statement of Corollary~\ref{close1} to prove Theorem~\ref{reduced} in the following form. 

\begin{theorem}\label{red}
Let $\frac{1}{n}\ll\del\ll\eps\ll1$. If $G$ is a $(1-\del)$-dense, $k$-coloured graph with $v(G)=2^{k-1}n$, containing no monochromatic odd connected matching of order $\geq(1+\delta)n$, then for any choice of profile $x(G)$ of $G$, there exists some $x^\ast\in O^\ast$ such that 
$$
\|x(G)/n-x^\ast\|<\varepsilon.
$$  
\end{theorem}

First we need the following two colour Ramsey type result which is a direct consequence of the more general Theorem $1.8$ in~\cite{BenLuc}.

\begin{lemma}\label{2stab}
 Let $\frac{1}{n}\ll\del\ll\eps$. If $H$ is a $(1-\del)$-dense, 2-coloured graph with $v(H)\geq(\frac{3}{2}+\eps)n$, then $H$ contains a monochromatic connected matching of order $\geq(1+\del)n$.\qed
\end{lemma}

\begin{proof}[Proof of Theorem~\ref{red}]
Given $\eps>0$, let $\gamma=\gamma_{\ref{close1}}(\eps)$ and $\del'=\del_{\ref{2stab}}(\eps)$. Suppose that $\del<\min\{\gamma^2k^{-2} 2^{-4k}, \del'2^{-2k}\}$ and that $n\geq \max\{n_{\ref{2stab}}(\del'), \del^{-1}\}$. Let $G$ be a $k$-coloured graph as in the statement of Theorem~\ref{red}. Let $x(G)$ be any choice of profile for $G$ and let the corresponding profile partition be $(V_\tau: \tau\in\{0,1,\ast\}^k)$. Note that $\norm{x(G)/n}=2^{k-1}$ and by Proposition~\ref{constraints}, we have that $x(G)/n\in X(\sqrt{\delta}k2^{2k})\subseteq X(\gamma)$. By Corollary~\ref{close1} there exists an element $x^\ast\in O$ such that 
\begin{equation}\label{2col}
\norm{x(G)/n-x^*}<\eps.
\end{equation}
Suppose that $x^\ast\in O\backslash O^\ast$, then $x_\tau=2$ for some $\tau\in\{0,1,\ast\}^k$ such that $\omega(\tau)=2$. It follows from \eqref{2col} that $x(G)_\tau=\abs{V_\tau}>(2-\eps)n\geq(3/2+\eps)n$. Let $H=G[V_\tau]$. By the definition of $V_\tau$, $H$ is a 2-coloured graph. Moreover since $G$ has at most $\delta \binom{v(G)}{2}\leq \del'\binom{v(H)}{2}$ edges missing, the same is true for $H$. It follows by Lemma~\ref{2stab} that $H$ contains a monochromatic connected matching of order $\geq(1+\del')n>(1+\delta)n$. However, by the definition of $V_\tau=V(H)$, any monochromatic component of $H$ is contained in a non-bipartite monochromatic component of $G$. Thus $G$ contains a monochromatic odd connected matching of order $>(1+\delta)n$ contrary to assumption. We conclude that $x^\ast\in O^\ast$.
\end{proof}

\section{The Regularity Method}
In this section we discuss the tools and results we need from the regularity method. Our starting point is  Szemer{\'e}di's Regularity Lemma~\cite{Szem} which we discuss briefly now.

 Let $G$ be a graph and let $A, B$ be disjoint subsets of $V(G)$. We call
 $$
 d_G(A,B):=\frac{e_G(A,B)}{\abs{A}\abs{B}}
 $$
 the \emph{density} of the pair $(A,B)$. For $\delta>0$, we say that the pair $(A,B)$ is $\del$-regular with respect to $G$ if, for every $A'\subseteq A$ and $B'\subseteq B$ satisfying $\abs{A'}\geq\delta\abs{A}$ and  $\abs{B'}\geq\delta\abs{B}$, we have
$$
\abs{d_G(A',B')-d_G(A,B)}<\delta.
$$
If, for $d>0$, we also have that $\abs{N_G(a)\cap B}\geq d\abs{B}$ for all $a\in A$ and $\abs{N_G(b)\cap A}\geq d\abs{A}$ for all $b\in B$, then we say that $(A,B)$ is $(\del, d)$-super-regular with respect to $G$. If the graph $G$ is clear from the context we may omit it from the above notation. The following is a version of Szemer{\'e}di's Regularity Lemma that appears as Theorem $1.18$ in~\cite{KomSim}.

\begin{theorem}[Multicolour Regularity Lemma]\label{reg}
For all $\delta>0$ and $k, \ell\in \mathbb{N}$ there exists $L=L(\del,k,\ell)$ and $M=M(\del,k,\ell)$ such that the following holds. For all $k$-coloured graphs $G$ on at least $M$ vertices, $V(G)$ may be partitioned into sets $V_0, V_1 \ldots, V_t$ such that
\begin{itemize}
\item $\ell\leq t\leq L$;
\item $\abs{V_0}<\del v(G)$ and $\abs{V_1}=\abs{V_2}=\ldots=\abs{V_t}$;
\item apart from at most $\del\binom{t}{2}$ exceptional pairs, the pairs $(V_i, V_j)$, $1\leq i<j\leq t$, are $\del$-regular with respect to $G_s$ for $s=1,\ldots,k$.
\end{itemize}

\end{theorem}
 
 We now state some technical lemmas related to \L uczak's method of connected matchings. First we need a definition. (It might be useful at this point to recall Definition~\ref{def:conmat}.)
 
\begin{definition}\label{def:super}
Let $\del,d\in[0,1]$ and $q,m\geq1$ be integers.
\begin{itemize}
\item Let $F$ be a graph on vertex set $[q]$ and let $U_1, \ldots, U_q$ be disjoint sets of size $m$. We call a graph $H$ on vertex set $\bigcup_{i\in[q]}U_i$ a \emph{$(\del, m)$-regular blow-up} of $F$ if whenever $\{i,j\}\in E(F)$, we have that $(U_i, U_j)$ is a $\del$-regular pair.

\item If in addition to the above, $d(U_i, U_j)\geq d$ for each edge $\{i,j\}$ of $F$m then we say that $H$ has \emph{minimum density} $d$. 

\item Suppose that $F$ is a connected matching and $H$ is a $(\del, m)$-regular blow-up of $F$ with minimum density $d$. If for each matching edge $\{i,j\}$ of $F$, the pair $(U_i, U_j)$ is in fact $(\del, d)$-super-regular in $H$, then we say that $H$ is a \emph{$(\del, d, m)$-super-regular blow-up} of $F$.
\end{itemize}
\end{definition}
Versions of the following two lemmas abound in the literature (e.g.~\cite{Exact}, \cite{Lucz}), but here we give statements tailored to our needs. However since they are not new, we defer their proofs to the Appendix.
\begin{lemma}\label{embed}
Let $q\geq4$ and suppose that $\frac{1}{m}\ll\del\ll d$.  Let $F$ be a connected matching of order $q$ such that every vertex of $F$ is incident to a matching edge and let $H$ be a $(\del,d,m)$-super-regular blow-up of $F$. Then the following holds:  

If $i,j\in V(F)$ and there is an $ij$-path of length $r$ in $F$, then for every  pair of vertices $u\in U_i$, $w\in U_j$, there exists a $uw$-path of length $\ell$ in $H$ for each $3q\leq\ell\leq (1-6\del)qm$ such that $\ell\equiv r\pmod{2}$. 
\end{lemma}

\begin{lemma}\label{embed2}
Let $q\geq 4$ and let $\frac{1}{m}\ll\del\ll d$. Let $F$ be an odd connected matching of order $q$ and suppose that $H$ is a $(\del, m)$-regular blow-up of $F$ with minimum density $d$. Then $H$~contains a cycle of length $\ell$ for each odd $3q\leq\ell\leq(1-6\del)qm$.
\end{lemma}

We borrow the following fact. 
 \begin{fact}(\cite[Lemma 9]{Gya}). \label{simple}
 Let $H$ be a $(1-\del)$-dense graph on $t$ vertices. Then $H$ has a~subgraph $H'$ such that $v(H')\geq(1-\sqrt{\del})t$ and $\del(H')\geq(1-2\sqrt{\del})t$.
 \end{fact}
 
We will also need the following two standard facts whose proofs we omit it here.
 
\begin{fact}\label{slicing}
Let $0<\del\leq1/2$ and let $(A,B)$ be a $\del$-regular pair with density $d$. Suppose that $A'\subseteq A$, $B'\subseteq B$ such that $\abs{A'}\geq(1-\del)\abs{A}$, $\abs{B'}\geq(1-\del)\abs{B}$. Then $(A', B')$ is $2\del$-regular with density $d'>d-\del$. Moreover, if $(A,B)$ is in fact $(\del, \beta)$-super-regular for some $\beta>0$, then $(A',B')$ is $(2\del, \beta-\del)$-super-regular. \qed
\end{fact}

 \begin{fact}\label{super}
Let $0<\del\leq1/2$ and let $(A,B)$ be a $\del$-regular pair with density $d$. Then there exist $A'\subseteq A$, $B'\subseteq B$ such that $\abs{A'}=(1-\del)\abs{A}$, $\abs{B'}=(1-\del)\abs{B}$ and $(A', B')$ is $\left(2\del, d-2\del\right)$-super-regular.\qed
\end{fact}

\section{Proof of Main Theorem}
In this final section we prove our main result Theorem~\ref{conj}, and therefore also Theorem~\ref{main}. The idea is to invoke Theorem \ref{red} to show that the profile of a certain reduced graph is close in $\ell_1$-norm to the profile of a hypercube colouring. We then translate this information to show that the original graph is close in edit distance to a hypercube colouring.

The stability-type methods of this section require some care due to the plethora of extremal constructions. Luckily hypercube colourings share enough common features for these methods to be viable and surprisingly we require no case analysis. 

First let us state a result which is a corollary of a classical theorem of Bondy~\cite{bondy1971pancyclic}.
\begin{theorem}\label{thm:pan}
Let $G$ be a graph on at least $3$ vertices with minimum degree $>v(G)/2$, then $G$ is pancyclic i.e. $G$ contains cycles of all lengths $3\leq\ell\leq v(G)$.
\end{theorem}

\begin{proof}[Proof of Theorem~\ref{conj}]
Let $0<\eps<2^{-4k}$, let 
\begin{equation}\label{initial}
\eta<\delta\leq\min\left\{\frac{1}{9}\delta_{\ref{red}}^2(\eps), \delta_{\ref{embed}}\left(\frac{1}{k}\right),  \delta_{\ref{embed2}}\left(\frac{1}{k}\right)\right\}, 
\end{equation}
let $n_0\geq\max\{n_{\ref{red}}(\del), \del^{-1/2}\}$ and let $L=L_{\ref{reg}}(\del, k, 2^kn_0)$.
Let $n$ be odd with
\begin{equation}\label{eq:bign}
n\geq \max\{Lm_{\ref{embed}}(\del), Lm_{\ref{embed2}}(\del), M_{\ref{reg}}(\del, k, 2^kn_0)\}.
\end{equation}
Finally let $G$ be a $k$-coloured copy of $K_N$ where $N>(2^{k-1}-\eta)n$ and assume that 
 \begin{equation}\label{assumption}
 \text{$G$ contains no monochromatic copy of $C_n$.}\tag{\dag}
 \end{equation}
 Applying Theorem~\ref{reg} to $G$ we obtain a partition of $V(G)$ into sets $V_0, V_1 \ldots, V_{t_0}$ such that
\begin{enumerate}[(i)]
\item $2^kn_0\leq t_0\leq L$;
\item $\abs{V_0}<\del N$ and $\abs{V_1}=\abs{V_2}=\ldots=\abs{V_{t_0}}$;
\item apart from at most $\del\binom{t_0}{2}$ exceptional pairs, the pairs $(V_i, V_j)$, $1\leq i<j\leq t_0$, are $\del$-regular with respect to $G_s$ for $s=1,\ldots,k$.
\end{enumerate}
It follows that for $i\in[t_0]$,
\begin{equation}\label{blow}
m:=\abs{V_i}\geq\frac{(1-\del)N}{t_0}.
\end{equation}

We construct a \emph{reduced graph} $R_0$ with vertex set $\{1,...,t_0\}$ and edge set formed by pairs $\{u, w\}$ for which $(V_u, V_w )$ is $\del$-regular with respect to $G_i$ for $i=1,\ldots,k$. It follows from (iii) of the above that $R_0$ is $(1-\del)$-dense. Fact~\ref{simple} allows us to find a subgraph $R\subseteq R_0$ satisfying $v(R)\geq(1-\sqrt{\delta})t_0$ and $\delta(R)\geq (1-2\sqrt{\delta})t_0$. Let $t=v(R)$ and assume without loss of generality that $V(R)=\{1, \ldots, t\}$. We $k$-colour $R$ by colouring an edge $\{u,w\}$ with the least colour $i$ for which 
 \begin{equation}\label{blow2}
 d_{G_i}(V_u, V_w)\geq \frac{1}{k}. 
\end{equation}
 Let $t'=t/2^{k-1}$ and note that by \eqref{blow} and the definition of $t$,
 \begin{equation}\label{mt}
 mt'\geq(1-2\sqrt{\del})n.
 \end{equation}
 Suppose that $R$ contains a monochromatic odd connected matching $F$ of order $q\geq(1+3\sqrt{\delta})t'$. Then $G$ contains a monochromatic $(\delta,m)$-regular blow-up of $F$ with minimum density $d$ for some $d\geq1/k$ by \eqref{blow2}. Note that since $t_0\leq L$ we have $m\geq m_{\ref{embed2}}(\delta)$ by \eqref{blow} and \eqref{eq:bign}. It follows from Lemma~\ref{embed2} that $G$ contains a monochromatic copy of $C_n$ since $n$ is odd and
 $$
 3q\leq3L\leq n\leq(1-6\del)(1+3\sqrt{\del})mt'\leq(1-6\del)qm,
 $$
 contradicting \eqref{assumption}. We conclude that $R$ contains no such odd connected matching. Let $(W_\tau: \tau\in\{0,1,\ast\}^k)$ be a profile partition of $R$ and let $x(R)=(\abs{W_\tau}: \tau\in\{0,1,\ast\}^k)$ be the  corresponding profile. It follows by Theorem~\ref{red} that there exists $x^\ast\in O^\ast$ such that 
 \begin{equation}\label{decomp}
 \norm{x(R)/t'-x^\ast}<\eps.
 \end{equation}

This tells us a lot about the structure of $R$, indeed it is `close to' a hypercube colouring. The aim is to use this fact to eventually say the same for $G$. By the definition of $O^\ast$ we have that $$supp(x^\ast)=\mathcal{M}\subseteq\{0,1,\ast\}^k,$$ for some perfect matching $\mathcal{M}$ of the hypercube $Q_k$ and $x^\ast_\tau=1$ for all $\tau\in\mathcal{M}$. Let
$$
W=R\big\backslash\bigcup_{\tau\in\mathcal{M}}W_\tau.
$$
We will treat $W$ as a `leftover set' of vertices of $R$ and study only the structure of $R\backslash W$. Note that by \eqref{decomp} we have 
\begin{equation}\label{big}
(1-\eps)t'<\abs{W_\tau}<(1+\eps)t'\ \text{for all}\ \tau\in\mathcal{M},
\end{equation}
and so by removing at most $2\eps t'$ vertices from each part $W_\tau$, where $\tau\in\mathcal{M}$, and absorbing these removed vertices into $W$, we may assume that these parts $W_\tau$ all have the same size $>(1-\eps)t'$. Note that even after this absorption we have 
\begin{equation}\label{W}
\abs{W}=t-\sum_{\tau\in\mathcal{M}}\abs{W_\tau}<t-2^{k-1}(1-\eps)t'=\eps t.
\end{equation}
We make a couple of observations regarding the colouring of $R$ with respect to these vertex classes. For $j\in [k]$ we let 
\[
I_j=\{\tau\in\mathcal{M}: \tau_j=\ast\}.
\]

\begin{lemma}\label{clique}
Let $\tau\in I_j$. Then $R[W_\tau]$ is monochromatic in the colour $j$ and has minimum degree at least $(1-2^{k+1}\sqrt{\delta})\abs{W_\tau}$.
\end{lemma}
\begin{proof}
By the definition of the profile partition, for each colour $i\neq j$, each pair $v, w\in W_\tau$ must lie in the same vertex class in an induced bipartite subgraph of $R_{i}$. It follows that if $\{v,w\}\in E(R)$ then it cannot receive the colour $i$ and hence must receive colour $j$. Since $\delta(R)\geq(1-2\sqrt{\delta})t$ we have
$$
\delta(R[W_\tau])\geq \abs{W_\tau}-1-2\sqrt{\delta}t\geq(1-2^{k+1}\sqrt{\delta})\abs{W_\tau}
$$
where for the last inequality we used \eqref{big}.
\end{proof}

\begin{definition}
Let $\sigma, \tau\in\{0,1,\ast\}^k$. We denote the set $\{i\in[k]: \{\sigma_i, \tau_i\}=\{0,1\}\}$ by $\Delta(\sigma, \tau)$. We call $\abs{\Delta(\sigma, \tau)}$ the \emph{distance} between $\sigma$ and $\tau$ and denote it by $d(\sigma, \tau)$.
\end{definition}

\begin{lemma}\label{d1}
Let $\sigma, \tau\in \mathcal{M}$ be distinct, then 
\begin{enumerate}[(i)]
\item Each edge of $R[W_\sigma, W_\tau]$ receives a colour from the set $\Delta(\sigma, \tau)$;\label{d11}
\item $R[W_\sigma, W_\tau]$ has minimum degree $\geq(1-~2^{k+1}\sqrt{\delta})\abs{W_\sigma}$, in particular $R[W_\sigma, W_\tau]$ is connected and contains a perfect matching.\label{d12}
\end{enumerate}
\end{lemma}
\begin{proof}
Let $\sigma\in I_j$, $\tau\in I_\ell$. Suppose that $j\neq\ell$. By the definition of the profile partition, for each colour $i\notin\Delta(\sigma, \tau)$, each pair $v\in W_\sigma$, $w\in W_\tau$ must lie in either the same vertex class in an induced bipartite subgraph of $R_i$ or they lie in different connected components of $R_i$. It follows that if $\{v, w\}\in E(R)$ then it must receive a colour from $\Delta(\sigma, \tau)$. Similarly, if $j=\ell$ then each edge of $R[W_\sigma, W_\tau]$ must receive a colour from $\Delta(\sigma, \tau)\cup\{j\}$. However, by Lemma~\ref{clique},  $R[W_\tau]$ and $R[W_\sigma]$ are both monochromatic in the colour $j$ and both have minimum degree at least $(1-2^{k+1}\sqrt{\delta})\abs{W_\sigma}>\abs{W_\sigma}/2$ (recall that $\abs{W_\sigma}=\abs{W_\tau}$). It follows, by Theorem \ref{thm:pan} for example, that $R[W_\tau]$ and $R[W_\sigma]$ are both Hamiltonian and non-bipartite. Using \eqref{big}, we deduce that if an edge of  $R[W_\sigma, W_\tau]$ receives the colour $j$, then $R$ contains a monochromatic odd connected matching in the colour $j$ of order at least
$$\abs{W_\sigma}+\abs{W_\tau}-2\geq2(1-\eps)t'-2 \geq(1+3\sqrt{\delta})t',$$
 which we showed previously was not the case. Part $(i)$ of the lemma follows. If $v\in W_\tau$ then, since $\delta(R)\geq(1-2\sqrt{\delta})t$, we have
$$
\abs{N(v)\cap W_\sigma}\geq\abs{W_\sigma}-2\sqrt{\delta}t\geq(1-2^{k+1}\sqrt{\delta})\abs{W_\sigma}.
$$
Similarly if $w\in W_\sigma$, then $\abs{N(w)\cap W_\tau}\geq(1-2^{k+1}\sqrt{\delta})\abs{W_\sigma}$. Since $1-2^{k+1}\sqrt{\delta}>1/2$, it follows that $R[W_\sigma, W_\tau]$ is connected and (e.g. by Hall's theorem) contains a perfect matching.

\end{proof}
Let $\Gamma$ denote the $k$-coloured multigraph on vertex set $\mathcal{M}$ where we have an edge between $\sigma$ and $\tau$ in each colour $j$ for which $R[W_\sigma, W_\tau]$ contains an edge of colour $j$. Note that since $\del(R)>(1-2\sqrt{\del})t$ and $\abs{W_\sigma}=\abs{W_\tau}>(1-\eps)t'> 2\sqrt{\del}t$, $R[W_\sigma, W_\tau]$ always contains an edge. Let $\Gamma^{\ast}$ denote the subgraph of $\Gamma$ where we keep only those edges that occur as the unique edge between a given pair of vertices in $\Gamma$. Recall that for $j\in [k]$, $\Gamma_j$, $\Gamma^\ast_j$ denote the $j$th colour class of $\Gamma$, $\Gamma^\ast$ respectively.

\begin{lemma}\label{H}
For each $j\in [k]$, the vertices of $\Gamma^\ast_j$ can be covered by a matching $T_j\subseteq\Gamma^\ast_j$ and the set $I_j$. Moreover $I_j$ is a set of isolated vertices in $\Gamma_j$.
\end{lemma}
\begin{proof}
Fix $j\in[k]$. If $\sigma\in I_j$ then $\sigma$ is an isolated vertex in $\Gamma_j$ by Lemma~\ref{d1}\eqref{d11}. If $\sigma\notin I_j$ then we may assume without loss of generality that $\sigma_j=0$. Let $\sigma'$ be the element of $\{0,1,\ast\}^k$ such that $\sigma'_j=1$ and $\sigma'_i=\sigma_i$ for all $i\neq j$. Let $H$ be the graph on $\mathcal{M}$ with edge set $\{\{\rho, \pi\}: \Delta(\rho, \pi)=\{j\}\}$. By Lemma~\ref{d1}\eqref{d11} we have $H\subseteq \Gamma^\ast_j$. The neighbours of $\sigma$ in $H$ are precisely those elements of $\mathcal{M}$ that are indistinguishable from $\sigma'$ i.e. those elements of $\mathcal{M}$ (viewed as edges of $Q_k$) that intersect $Q(\sigma')$. Since $\mathcal{M}$ is a perfect matching of $Q_k$, and $\abs{Q(\sigma')}=2$, there are either 1 or 2 such elements of $\mathcal{M}$. It follows that $H$ is the disjoint union of cycles (where we consider an edge a cycle) and the independent set $I_j$. Since $H$ is bipartite with bipartition $\{\tau\in\mathcal{M}: \tau_j=0\ \text{or}\ \ast\}\cup\{\tau\in\mathcal{M}: \tau_j=1\}$, the cycles in $H$ are all even. The result follows.
\end{proof}

Let $j\in[k]$, then for each $\{\sigma, \tau\}\in T_j$ ($T_j$ as in the statement of Lemma~\ref{H}), we may fix a monochromatic perfect matching $M_{\sigma\tau}^j$ in the colour $j$ in $R[W_\sigma, W_\tau]$ by Lemmas~\ref{d1}\eqref{d12} and \ref{H}. Let 
\[
\mathcal{T}_j=\bigcup_{\{\sigma, \tau\}\in T_j}M_{\sigma\tau}^j.
\]
and note that $\mathcal{T}_j$ is a matching in $R$, monochromatic in the colour $j$, which covers the vertex set $\bigcup_{\tau\notin I_j}W_\tau$. The following corollary hints at an important common feature of all hypercube colourings.

\begin{corollary}\label{cor:common}
Given $j\in[k]$ and $\rho\in\mathcal{M}\backslash I_j$, there exists $\pi\in\mathcal{M}$ such that $R[W_\rho,W_\pi]$ contains a monochromatic connected perfect matching in the colour $j$ whose matching edges are edges of $\mathcal{T}_j$.
\end{corollary}
\begin{proof}
Since $\rho\notin I_j$, by Lemma \ref{H} there must exist $\pi\in\mathcal{M}$ such that $\{\rho,\pi\}$ is an edge of $T_j\subseteq\Gamma^\ast_j$. By the definition of $\Gamma^\ast$, we have that $R[W_\rho, W_\pi]$ is monochromatic in the colour $j$. The result follows from the definition of $\mathcal{T}_j$ and Lemma~\ref{d1}\eqref{d12}.
\end{proof}

It will be useful to prune the sets $V_i$ for $i\in R$ in such a way that if $\{x,y\}$ is an edge of the matching $\mathcal{T}_j$ then $G_j[V_x, V_y]$ is super-regular. 

\begin{lemma}\label{lem:supersub}
For each $i\in R$ there exists $V_i'\subseteq V_i$ such that 
\begin{enumerate}[(i)]
\item $\abs{V_i'}=(1-2^k\del)m$ for all $i\in R$,
\item $G_j[V'_x, V'_y]$ is $2^{k+1}\del$-regular with density $\geq\frac{1}{k+1}$ for all $j\in[k], \{x,y\}\in R_j$,
\item $G_j[V'_x, V'_y]$ is $(2^{k+1}\del, \frac{1}{k+1})$-super-regular for all $j\in[k], \{x,y\}\in\mathcal{T}_j$.
\end{enumerate}
\end{lemma}
\begin{proof}
For $i\in R$ we define a sequence of subsets $V_i^0, \ldots, V_i^k$ of $V_i$ recursively. Let $V_i^0:=V_i$ for all $i\in R$. Suppose that for all $i\in R$ we have found $V_i^\ell\subseteq V_i$ with the following properties.
\begin{enumerate}[(a)]
\item $\abs{V_i^\ell}\geq(1-2^\ell\del)m$ for all $i\in R$,\label{aa}
\item $G_j[V^\ell_x, V^\ell_y]$ is $2^\ell\del$-regular with density $\geq1/k-2^{\ell}\del$ for all $j\in[k], \{x,y\}\in R_j$,\label{bb}
\item $G_j[V^\ell_x, V^\ell_y]$ is $(2^\ell\del, 1/k-2^{\ell+1}\del)$-super-regular for all $j\in[\ell], \{x,y\}\in\mathcal{T}_j$.\label{cc}
\end{enumerate}
By Fact~\ref{super}, for each edge $\{u,w\}$ in the matching $\mathcal{T}_{\ell+1}$ (so in particular $\{u,w\}\in R_{\ell+1}$) there exists $V^{\ell+1}_u\subseteq V^{\ell}_u$ and $V^{\ell+1}_w\subseteq V^{\ell}_w$ such that $\abs{V^{\ell+1}_u}=(1-2^{\ell}\del)\abs{V^{\ell}_u}$, $\abs{V^{\ell+1}_w}=(1-2^{\ell}\del)\abs{V^{\ell}_w}$ and $G_{\ell+1}[V^{\ell+1}_u, V^{\ell+1}_w]$ is $(2^{\ell+1}\del, 1/k-2^{\ell+2}\del)$-super-regular. If $i$ is not incident to any edge of $\mathcal{T}_{\ell+1}$ then simply set $V_i^{\ell+1}=V_i^\ell$. Note that by \eqref{aa}, for all $i\in R$,
\[
\abs{V_i^{\ell+1}}\geq(1-2^{\ell}\del)\abs{V^{\ell}_i}\geq(1-2^{\ell}\del)^2m\geq(1-2^{\ell+1}\del)m.
\]
For $j\in[k]$ and $\{x,y\}\in R_j$, by \eqref{bb} and Fact \ref{slicing} we have that $G_j[V^{\ell+1}_x, V^{\ell+1}_y]$ is $2^{\ell+1}\del$-regular with density $\geq1/k-2^{\ell+1}\del$. Using \eqref{cc} and Fact \ref{slicing}, it also follows that $G_j[V^{\ell+1}_x, V^{\ell+1}_y]$ is $(2^{\ell+1}\del, 1/k-2^{\ell+2}\del)$-super-regular for all $j\in[\ell], \{x,y\}\in\mathcal{T}_j$. We have shown that the sets $V_i^{\ell+1}, i\in R$, satisfy \eqref{aa}-\eqref{cc} (with $\ell$ replaced by $\ell+1$). The result follows by letting $V_i'$ be any subset of $V_i^k$ of size $(1-2^k\del)m$ for all $i\in R$ and appealing to Fact \ref{slicing}, noting that $1/(k+1)\leq 1/k-2^{k+2}\del$. 

\end{proof}

Given $\sigma\in \mathcal{M}$, let 
\[
\widetilde{W}_\sigma=\bigcup_{i\in W_\sigma}V'_i\subseteq V(G),
\]
and let
\[
\widetilde{W}=V(G)\big\backslash\bigcup_{\tau\in\mathcal{M}}\widetilde{W}_\tau.
\]
As with $W$, we think of $\widetilde{W}$ as a small leftover set of vertices. Let $m':=(1-2^{k}\del)m$ and note that by \eqref{mt}, $m't'\geq(1-3\sqrt{\del})n$ and so by \eqref{big}
\begin{equation}\label{m}
\abs{\widetilde{W_\tau}}\geq(1-\eps)m't'\geq(1-2\eps)n\ \text{for all }\tau\in\mathcal{M}.
\end{equation}
We also have
\begin{equation}\label{Wtilde}
\abs{\widetilde{W}}\leq N-2^{k-1}(1-\eps)m't'=N-(1-\eps)(1-2^{k}\del)mt\leq2\eps N. 
\end{equation}
Where for the last inequality we recalled \eqref{blow}.

We can now establish our first piece of structure on the graph $G$. We show that almost all of $V(G)$ can be covered by $2^{k-1}$ monochromatic cliques of equal size. First we make a quick definition.
 \begin{definition}
If $\tau\in\{0,1,\ast\}^k$ has weight 1, we let $c(\tau)$ denote the unique element of $i\in[k]$ for which $\tau_i=\ast$.
\end{definition}
 
\begin{lemma}\label{comp}
For all $\sigma\in\mathcal{M}$, $G[\widetilde{W}_\sigma]$ is monochromatic in the colour $c(\sigma)$.
\end{lemma}

\begin{proof}

Suppose that $G[\widetilde{W}_\sigma]$ contains an edge $\{x,y\}$ of colour $j\neq c(\sigma)$ (so that $\sigma\notin I_j$). By Corollary~\ref{cor:common} there exists $\tau\in\mathcal{M}$ such that $R_j[W_\sigma, W_\tau]$ contains a connected perfect matching, $F$, whose matching edges are edges of $\mathcal{T}_j$. Let $q:=v(F)$, then by \eqref{big} we have $2(1-\eps)t'\leq q\leq2(1+\eps)t'$. By Lemma \ref{lem:supersub} we see that $G_j[\widetilde{W}_\sigma, \widetilde{W}_\tau]$ contains a spanning $(2^{k+1}\del, 1/(k+1), m')$-super-regular blow-up of $F$ (with the $V'_i$ playing the role of the $U_i$ in Definition \ref{def:super}). Suppose that $x\in V'_a$ and $y\in V'_b$, then $a$ and $b$ lie on the same side of the bipartition of the connected graph $F$ and so $F$ contains an $ab$-path of even length. Note that $m'\geq n_{\ref{embed}}(\del)$, by \eqref{blow} and \eqref{eq:bign}. We may therefore apply Lemma~\ref{embed} to deduce that $G_j[\widetilde{W}_\sigma, \widetilde{W}_\tau]$ contains a path of length $n-1$ joining $x$ and $y$ since $n-1\equiv 0 \pmod{2}$ and
\begin{equation}\label{eq:asin}
(1-6\cdot2^k\del)qm'\geq2(1-6\cdot2^k\del)(1-\eps)m't'\geq n-1\geq3L\geq3q,
\end{equation}
where we used \eqref{initial}, \eqref{eq:bign} and \eqref{m}. Together with the edge $\{x,y\}$ this creates a monochromatic copy of $C_n$ in $G$, contrary to assumption \eqref{assumption}.
\end{proof}
Our aim now is to say something about the edges of $G$ lying between $\widetilde{W}$ and the rest of the graph (see Lemma~\ref{fin} below). With the multigraph $\Gamma$ in mind, we make the following definition. 

\begin{definition}\label{def:admis}
Let $\mathcal{M}'$ be a perfect matching of $Q_k$ and let $\varphi : \mathcal{M}\to\mathcal{M}'$ be a bijection such that $c(\varphi(\tau))=c(\tau)$ for all $\tau\in\mathcal{M}$. Suppose that $\Psi$ is a $k$-coloured multigraph on vertex set $\mathcal{M}$. We call $\varphi$ an \emph{admissible labelling} of $\Psi$ if for all $\sigma, \tau\in\mathcal{M}$, the edges between $\sigma, \tau$ in $\Psi$ only take colours from the set $\Delta(\varphi(\sigma),\varphi( \tau))$.
\end{definition}

Note that by Lemma~\ref{d1}\eqref{d11} the identity map $\iota : \mathcal{M}\to\mathcal{M}$ is an admissible labelling of $\Gamma$. The following Lemma gives a useful way of generating new admissible labellings of $\Gamma$. For $\tau\in\{0,1,\ast\}^k$, such that $\tau_j\neq\ast$, we let $\tau^j:=(\tau_1,\ldots,\tau_{j-1},1-\tau_j,\tau_{j+1},\ldots,\tau_k)$  i.e. $\tau^j$ denotes $\tau$ with the $j$th coordinate flipped.

\begin{lemma}\label{admis}
Let $\varphi$ be an admissible labelling of $\Gamma$. Let $j\in[k]$ and let $C$ be the vertex set of a component of $\Gamma_j$ such that $\tau_j\neq\ast$ for all $\tau\in C$. Let $\varphi'$ be the function on $\mathcal{M}$ given by $\varphi'(\tau)=\varphi(\tau)^j$ for all $\tau\in C$, $\varphi'(\tau)=\varphi(\tau)$ otherwise. Then $\varphi'$ is an admissible labelling of $\Gamma$.
\end{lemma}

\begin{proof}
First note that by the definition of $\varphi'$ and the fact that $\varphi$ is admissible, each element of $\varphi'(\mathcal{M})$ has weight $1$ and $c(\varphi'(\tau))=c(\varphi(\tau))=c(\tau)$ for all $\tau\in\mathcal{M}$. Let us check that the image of $\varphi'$ is a perfect matching of $Q_k$ (i.e. a distinguishable set of size $2^{k-1}$). It suffices to show that if $\sigma, \tau\in\mathcal{M}$ are distinct, then $\varphi'(\sigma), \varphi'(\tau)$ are distinguishable (i.e. $\Delta(\varphi'(\sigma),\varphi'(\tau))\neq\emptyset$). We do this by considering an edge between $\sigma$ and $\tau$ in $\Gamma$ and showing that if it has the colour $i$ then $i\in\Delta(\varphi'(\sigma),\varphi'(\tau))$. Note that this in fact suffices to show that $\varphi'$ is admissible.

Suppose then that there is an edge between $\sigma, \tau$ in $\Gamma$ in the colour $i$. Since $\varphi$ is admissible we have $i\in\Delta(\varphi(\sigma),\varphi(\tau))$. Suppose that $i\neq j$ then by the definition of $\varphi'$, $\varphi'(\tau)_i=\varphi(\tau)_i$ and $\varphi'(\sigma)_i=\varphi(\sigma)_i$ and so $i\in\Delta(\varphi'(\sigma),\varphi'(\tau))$ also. Suppose then that $i=j$, so that either $\sigma, \tau\in C$ or $\sigma, \tau\in\mathcal{M}\backslash C$. If $\sigma, \tau\in C$, then $\varphi'(\tau)_i=1-\varphi(\tau)_i$, $\varphi'(\sigma)_i=1-\varphi(\sigma)_i$ and if $\sigma, \tau\in\mathcal{M}\backslash C$, then $\varphi'(\tau)_i=\varphi(\tau)_i$, $\varphi'(\sigma)_i=\varphi(\sigma)_i$. In either case $i\in\Delta(\varphi'(\sigma),\varphi'(\tau))$.\\
\end{proof}

The following lemma allows us to associate each vertex in $\widetilde{W}$ to some class $\widetilde{W}_\sigma$ in $G$.

\begin{lemma}\label{fin}
Let $v\in \widetilde{W}$. Then there exists $\sigma\in \mathcal{M}$ such that $G[v, \widetilde{W}_\sigma]$ is monochromatic in the colour $c(\sigma)$.
\end{lemma}
\begin{proof}
Suppose otherwise, then for each $\sigma\in \mathcal{M}$ there exists a $u\in\widetilde{W}_\sigma$ such that the edge $\{v,u\}$ receives a colour $j_\sigma\neq c(\sigma)$. We augment the multigraph $\Gamma$ in the following way. We add the vertex $v$ to $\Gamma$ and for each $\sigma\in\mathcal{M}$ we add an edge between $v$ and $\sigma$ in the colour $j_\sigma$. Let us call this augmented multigraph $\Gamma^+$.\\

\begin{claim}
$\Gamma^+$ contains a monochromatic odd cycle.
\end{claim}
\begin{proof}[Proof of Claim]
Suppose otherwise and choose an admissible labelling $\varphi$ of $\Gamma$ that minimises the function 
$$
S(\varphi)=\sum_{i\in[k]}\abs{\{\tau\in\mathcal{M}: j_\tau=i\text{ and } \varphi(\tau)_i=1\}}.
$$ Suppose that $S(\mathcal{\varphi})>0$, then there exists a colour $j\in[k]$ and an element $\sigma\in \mathcal{M}$ for which $j_\sigma=j$ and $\varphi(\sigma)_j=1$. Let $C$ denote the component of $\Gamma_j$ containing the vertex $\sigma$ and note that by the definition of admissibility $C$ is bipartite with parts $\{\tau\in C : \varphi(\tau)_j=0\}$ and $\{\tau\in C : \varphi(\tau)_j=1\}$. Note that since $C$ is connected in $\Gamma_j$ this is the unique bipartition of $C$. Since $\Gamma^+_j$ is bipartite by assumption we must therefore have that $\varphi(\tau)_j=1$ for all $\tau\in C$ such that $j_\tau=j$. Let $\varphi'$ denote the function on $\mathcal{M}$ given by $\varphi'(\tau)=\varphi(\tau)^j$ for all $\tau\in C$, $\varphi'(\tau)=\varphi(\tau)$ otherwise. By Lemma~\ref{admis}, $\varphi'$ is an admissible labelling of $\Gamma$, however $S(\varphi')<S(\varphi)$ contradicting the minimality of $\varphi$. We conclude that $S(\varphi)=0$ i.e.
\begin{equation}\label{foreach}
\text{For all $i\in[k]$, $\tau\in \mathcal{M}$, if $j_\tau=i$ then $\varphi(\tau)_i=0$.}
\end{equation}
Since $\varphi(\mathcal{M})$ is a perfect matching of $Q_k$, there must exist $\rho\in\mathcal{M}$ such that the edge $\varphi(\rho)$ is incident to the vertex $(1,1,\ldots,1)$ (formally $Q(\varphi(\rho))$ contains $(1,1,\ldots,1)$). Without loss of generality suppose $\varphi(\rho)=(\ast, 1, \ldots, 1)$. However, whatever value $j_\rho$ takes, we contradict \eqref{foreach}. This concludes the proof of the claim.
\renewcommand{\qedsymbol}{\footnotesize{$\square$}}
\end{proof}

Suppose that $\Gamma^+$ contains a monochromatic odd cycle in the colour $j$. Since $\Gamma_j$ is bipartite and $\Gamma^+_j$ is not, there must exist $\sigma, \tau\in \mathcal{M}$ such that $\sigma, \tau$ lie in opposite parts of the bipartition of a connected component in $\Gamma_j$ and the edges $\{v, \sigma\}, \{v, \tau\}$ both have colour $j$ in $\Gamma^+$. By the definition of $\Gamma^+$, there exist vertices $u\in\widetilde{W}_\sigma$, $w\in\widetilde{W}_\tau$ such that $\{v, u\}$ and $\{v,w\}$ both have colour $j$ in $G$ and $j\neq c(\sigma)$ or $c(\tau)$ i.e. $\sigma, \tau\notin I_j$. Suppose that $u\in V'_a$ and $w\in V'_b$ then by Lemmas~\ref{d1}\eqref{d12} and \ref{H} and the definition of $\mathcal{T}_j$, $a$ and $b$ lie in opposite parts of a bipartite connected matching, $F$, in $R_j$ whose matching edges span $F$ and are edges of $\mathcal{T}_j$ (in particular there is an $ab$-path of odd length in $F$). Moreover we may assume that $F$ spans the vertex sets $W_\sigma, W_\tau$ in $R_j$ and so by \eqref{big}, $2(1-\eps)t'\leq v(F)\leq v(R)\leq L$. By Lemma \ref{lem:supersub}, we have a $(2^{k+1}\del, 1/(k+1), m')$-super-regular blow-up of $F$ in $G_j$. By Lemma~\ref{embed} (using inequalities as in \eqref{eq:asin} and noting that $n-2$ is odd) there exists a path of length $n-2$ joining $u$ and $w$ in $G_j$. This together with the edges $\{v, u\}$, $\{v,w\}$ forms a monochromatic copy of $C_n$ in $G$ contrary to assumption \eqref{assumption}. This concludes the proof of Lemma~\ref{fin}.
\end{proof}

Using Lemma~\ref{fin} we may define a function $f: \widetilde{W}\to\mathcal{M}$ where $f(v)$ is an element of $\mathcal{M}$ such that $G[v, \widetilde{W}_{f(v)}]$ is monochromatic in the colour $c(f(v))$. For each $\tau\in\mathcal{M}$, let $U_\tau=\widetilde{W}_\tau\cup f^{-1}(\{\tau\})$. By \eqref{m}, \eqref{Wtilde}, \eqref{blow} and Lemma~\ref{comp} we have that 
\begin{equation}\label{U}
\delta(G_{c(\tau)}[U_\tau])\geq(1-2^{k+1}\eps)\abs{U_\tau}\ \text{for all }\tau\in\mathcal{M}.
\end{equation}
Note that the sets $U_\tau$, $\tau\in\mathcal{M}$, partition the vertex set of $G$ and so if $N\geq 2^{k-1}(n-1)+1$ then by the pigeonhole principle there exists $\sigma\in\mathcal{M}$ such that $\abs{U_\sigma}\geq n$. However, by \eqref{U} and Theorem \ref{thm:pan}, it follows that $U_\sigma$ contains a monochromatic copy of $C_n$ in the colour $c(\sigma)$, contrary to assumption \eqref{assumption}. We therefore have that $N\leq2^{k-1}(n-1)$. Note that at this point we have done enough to prove Theorem~\ref{main}.

 It remains to show that $G$ is close in edit distance to a hypercube colouring. Recall that $\abs{\widetilde{W}}\leq2\eps N$ and so there are at most $2\eps N^2$ edges of $G$ incident to $\widetilde{W}$. We now aim to show that $G\backslash\widetilde{W}$ is close to a hypercube colouring. Recall that we have partitioned the vertex set of $G\backslash\widetilde{W}$ into the monochromatic, equally sized cliques $\{\widetilde{W}_\tau:\tau\in\mathcal{M}\}$. For $\sigma\in\mathcal{M}$, we showed that $\abs{\widetilde{W}_\sigma}\geq(1-2\eps)n$ and $\widetilde{W}_\sigma$ is monochromatic in the colour $c(\sigma)$. First note that at most $2\eps n$ vertices of $G\backslash\widetilde{W}_\sigma$ have more than $2\eps n$ neighbours in $\widetilde{W}_\sigma$ in the colour $c(\sigma)$ else we immediately find a monochromatic $C_n$ in the colour $c(\sigma)$ in $G$. It follows that there are at most $2\eps n N$ edges leaving the clique $\widetilde{W}_\sigma$ in the colour $c(\sigma)$. Over all $\tau\in\mathcal{M}$, there are therefore at most $2^k\eps n N<3\eps N^2$ edges in total leaving a clique $\widetilde{W}_\tau$ in the colour $c(\tau)$.

Let $\Phi$ now be the multigraph on vertex set $\mathcal{M}$ where we have an edge between $\sigma$ and $\tau$ in the colour $j$ for each $j\notin\{c(\sigma),c(\tau)\}$ for which $G[\widetilde{W}_\sigma,\widetilde{W}_\tau]$ contains a matching of two edges in the colour $j$. First we observe that to complete the proof it suffices to show that there exists an admissible labelling $\varphi$ of $\Phi$ (recall Definition \ref{def:admis}). Indeed suppose that this is the case, then since $\varphi$ is admissible, for each pair of distinct $\sigma, \tau\in\mathcal{M}$ and each $j\notin\Delta(\varphi(\sigma), \varphi(\tau))\cup\{c(\sigma),c(\tau)\}$, we have that $G_j[\widetilde{W}_\sigma,\widetilde{W}_\tau]$ contains no matching of two edges and hence contains at most $\abs{\widetilde{W}_\sigma}<n$ edges in total. It follows that there is a hypercube colouring $H$ associated to the perfect matching $\varphi(\mathcal{M})$ of $Q_k$, where $H$ has vertex set $V(G)\backslash\widetilde{W}$, such that for each $i\in[k]$,
\[
\abs{G_i\triangle H_i}\leq2\eps N^2+\abs{(G\backslash\widetilde{W})_i\triangle H_i}\leq 2\eps N^2+3\eps N^2+n\binom{2^{k-1}}{2}\leq6\eps N^2.
\]
The $2\eps N^2$ term accounts for edges of $G_i$ incident to $\widetilde{W}$, the $3\eps N^2$ term accounts for edges of $G_i$ leaving a clique $\widetilde{W}_\tau$ where $c(\tau)=i$, and the $n\binom{2^{k-1}}{2}$ term accounts for edges of $G_i$ lying between pairs $\widetilde{W}_\tau, \widetilde{W}_\sigma$ for which $i\notin\Delta(\varphi(\sigma), \varphi(\tau))\cup\{c(\sigma),c(\tau)\}$. We have thus shown that $G$ is $6\eps$-close to $H$. It remains to show that we have the desired labelling of $\Phi$.
\begin{claim}\label{claim:mono}
$\Phi$ contains no monochromatic odd cycle.
\end{claim}
\begin{proof}[Proof of Claim]
Suppose otherwise and let $\sigma_1\ldots\sigma_\ell$ be an odd cycle in $\Phi$ in the colour $j$. This allows us to fix a matching of size two in graphs $G_j[\widetilde{W}_{\sigma_i}, \widetilde{W}_{\sigma_{i+1}}]$ for $i=1,\ldots,\ell$ (where $\sigma_{\ell+1}:=\sigma_1$). Let $S$ be the subset of vertices of $G$ saturated by these matchings and note that $\abs{S}<2^{k+1}$. We first aim to build a short even path in $G_j$ with endpoints in $\widetilde{W}_{\sigma_1}$ and $\widetilde{W}_{\sigma_\ell}$. 

Let $x\in S\cap \widetilde{W}_{\sigma_1}$ and suppose that for some $2\leq r<\ell$ there exists $y\in \widetilde{W}_{\sigma_r}$ such that $G_j$ contains an $xy$-path $P_r$ of length $r-1+2L(r-2)$ where $\abs{P_r\cap S\cap\widetilde{W}_{\sigma_r}}=1$ and $P_r\cap S\cap\widetilde{W}_{\sigma_s}=\emptyset$ for $r<s\leq\ell$ (note that this does indeed hold for $r=2$). We may then pick $w\in\widetilde{W}_{\sigma_r}\cap S$ and $z\in \widetilde{W}_{\sigma_{r+1}}\cap S$ such that $\{w,z\}$ is an edge of $G_j[\widetilde{W}_{\sigma_{r}}, \widetilde{W}_{\sigma_{r+1}}]$ and $w\neq y$ (here we are using that we have a matching of size two available to us by the definition of $\Phi$). By the definition of $\Phi$, $\sigma_r$ is not in $I_j$ and so by Corollary \ref{cor:common} there exists $\pi\in\mathcal{M}$ such that $R_j[W_{\sigma_r}, W_\pi]$ contains a connected perfect matching, $F$, whose matching edges are edges of $\mathcal{T}_j$. By Lemma \ref{lem:supersub} we see that $G_j[\widetilde{W}_{\sigma_r}, \widetilde{W}_\pi]$ contains a spanning $(2^{k+1}\del, 1/(k+1), m')$-super-regular blow-up of $F$ where $2(1-\eps)t'\leq v(F)\leq2(1+\eps)t'$ by \eqref{big}. Moreover $w$ and $y$ lie in the same part in the bipartition of this blow-up. By calculations similar to those made previously, we may apply Lemma~\ref{embed} to deduce that $G_j[\widetilde{W}_{\sigma_2}, \widetilde{W}_\pi]$ contains an $yw$-path $Q$ of length $2L$. Moreover, since $\abs{P_r\cup S}\leq2^kL$ and using Fact~\ref{slicing} it is easy to ensure that $Q$ only intersects $P_r\cup S$ at its endpoints. It follows that $P_{r+1}:=P_rQz$ is an $xz$-path of length $r+2L(r-1)$ where $\abs{P_{r+1}\cap S\cap\widetilde{W}_{\sigma_{r+1}}}=1$ and $P_{r+1}\cap S\cap\widetilde{W}_{\sigma_s}=\emptyset$ for $r+1<s\leq\ell$. It follows by recursion that there exists $u\in\widetilde{W}_{\sigma_\ell}$ and an $xu$-path $P_\ell$ of length $p:=\ell-1+2L(\ell-2)$ and $\abs{P_\ell\cap S\cap\widetilde{W}_{\sigma_{\ell}}}=1$. Note that the length of $P_\ell$ is even.

Finally, let $\{v,t\}\subseteq S$ be an edge in $G_j[\widetilde{W}_{\sigma_\ell}, \widetilde{W}_{\sigma_1}]$  where $v\in\widetilde{W}_{\sigma_\ell}$ and $v\neq u$. If $x=t$ then applying Lemma~\ref{embed} as above we find a $uv$-path $Q_0$ in the colour $j$ of length $n-p-1$ intersecting $P_\ell\cup S$ only at its endpoints. It follows that $P_\ell Q_0x$ is a monochromatic copy of $C_n$ contradicting \eqref{assumption}. Similarly, if $x\neq t$, we find a $uv$-path $Q_1$ of length $2L$ and a $tx$-path $Q_2$ of length $n-p-2L-1$ both in the colour $j$ so that $P_\ell Q_1t Q_2$ is a monochromatic copy of $C_n$ contradicting \eqref{assumption}.
\renewcommand{\qedsymbol}{\footnotesize{$\square$}}
\end{proof}

We now construct an admissible labelling of $\Phi$ recursively. Suppose that $\mathcal{M}'$ is a perfect matching of $Q_k$ and that $\psi : \mathcal{M}\to\mathcal{M}'$ is some bijection. Let $\sigma, \tau\in\mathcal{M}$ and suppose there is an edge $f$ in $\Phi$ between $\sigma$ and $\tau$ with colour $j$ not in $\Delta(\psi(\sigma),\psi(\tau))$. We will call such an edge `bad' (with respect to $\psi$).

Let $\{f_1, \ldots, f_t\}$ be the set of edges of $\Phi$ that are bad with respect to the identity map $\iota: \mathcal{M}\to\mathcal{M}$ and note that $\iota$ is an admissible labelling of $\Phi\backslash\{f_1, \ldots, f_t\}$. Suppose now that $\varphi_i$ is an admissible labelling of $\Phi^i:=\Phi\backslash\{f_1,\ldots, f_i\}$ for some $1\leq i\leq t$. Suppose that $f_i$ is bad with respect to $\varphi_i$ and that $f_i$ has colour $j$ and lies between $\sigma, \tau\in\mathcal{M}$. Note that $j\notin\{c(\sigma), c(\tau)\}$ by the definition of $\Phi$. Moreover by the admissibility of $\varphi_{i}$ we have $c(\sigma)=c(\varphi_{i}(\sigma))$ and $c(\tau)=c(\varphi_{i}(\tau))$. Since $f_i$ is bad it follows that we must have $\varphi_i(\sigma)_j=\varphi_i(\tau)_j\in\{0,1\}$. Let us show that $\sigma, \tau$ lie in separate components of $\Phi^i_j$ (the $j$th colour class of $\Phi^i$). Suppose otherwise and take a path in $\Phi^i_j$ joining $\sigma$ and $\tau$. Since $\varphi_i$ is admissible for $\Phi^i$ and $\varphi_i(\sigma)_j=\varphi_i(\tau)_j$ this path must have even length. It follows that $f_i$ completes this path to a monochromatic odd cycle in $\Phi$ contradicting Claim \ref{claim:mono}. Let $C$ then denote the component of $\Phi^i_j$ containing $\tau$ (so that $\sigma\notin C$). Let $\varphi_{i-1}$ be the function on $\mathcal{M}$ given by $\varphi_{i-1}(\tau)=\varphi_{i}(\tau)^j$ for all $\tau\in C$, $\varphi_{i-1}(\tau)=\varphi_{i}(\tau)$ otherwise. By Lemma~\ref{admis}, $\varphi_{i-1}$ is an admissible labelling of $\Phi^i$. Since 
\[
j\in\Delta(\varphi_{i}(\sigma),\varphi_{i}( \tau)^j)=\Delta(\varphi_{i-1}(\sigma),\varphi_{i-1}( \tau)),
\] 
we also have that $\varphi_{i-1}$ is an admissible labelling of $\Phi^{i-1}$. If $f_i$ is not bad with respect to $\varphi_{i}$ we simply let $\varphi_{i-1}=\varphi_{i}$. Running this recursion to the end we obtain an admissible labelling $\varphi_0$ of $\Phi$ as required.

\end{proof}

\noindent
\textbf{Concluding Remarks.} A simple adaptation of the proof method in this paper proves the following generalisation of Theorem~\ref{main}.

\begin{theorem}
For all $k\geq3$ there exists $N_k$ such that the following holds. If $N_k\leq n_1\leq n_2\ldots\leq n_k$ are all odd then
$$
R(C_{n_1},\ldots,C_{n_k})=2^{k-1}(n_k-1)+1.
$$
\end{theorem}

 The off diagonal case has been well-studied: Erd\H{o}s \emph{et al.}~\cite{Erd} determined the value of $R(C_{n}, C_{\ell_1}, C_{\ell_2})$ and $R(C_{n}, C_{\ell_1}, C_{\ell_2}, C_{\ell_3})$ for $\ell_i$ fixed and $n$ sufficiently large. In a similar vein, as a corollory to a more general result in the study of Ramsey goodness, Allen, Brightwell and Skokan~\cite{Good} determined the value of $R(C_n, C_{\ell_1},\ldots, C_{\ell_k})$ for $\ell_i$ fixed and odd satisfying $\ell_i>2^i$ for $1\leq i\leq k$ and $n$ sufficiently large. In~\cite{FigLucz2}, Figaj and \L uczak asymptotically determine the Ramsey number of a triple of large cycles with any fixed combination of parities for the cycle lengths. In the case where not all of the cycles have the same parity,  Ferguson~\cite{ferguson2015ramsey1, ferguson2015ramsey2, ferguson2015ramsey3} strengthened the asymptotic results of~\cite{FigLucz2} to exact results. It would be interesting to extend the methods of the present paper to such a mixed parity setting. More generally, we would like to investigate whether the analytic approach presented here has wider applications in Ramsey theory. 
\\

\noindent
\textbf{Acknowledgements.} We would like to thank Julia B{\"{o}}ttcher for helpful discussions and comments.
\bibliographystyle{halpha} 
\bibliography{Exact}

\newcommand{\etalchar}[1]{$^{#1}$}
\begin{thebibliography}{B{\L}S{\etalchar{+}}12}

\bibitem[ABS13]{Good}
P.~Allen, G.~Brightwell, and J.~Skokan.
\newblock Ramsey-goodness---and otherwise.
\newblock {\em Combinatorica}, 33(2):125--160, 2013.

\bibitem[BE73]{BonErd}
J.~A. Bondy and P.~Erd{\H{o}}s.
\newblock Ramsey numbers for cycles in graphs.
\newblock {\em J. Combinatorial Theory Ser. B}, 14:46--54, 1973.

\bibitem[B{\L}S{\etalchar{+}}12]{BenLuc}
F.~Benevides, T.~{\L}uczak, A.~Scott, J.~Skokan, and M.~White.
\newblock Monochromatic cycles in 2-coloured graphs.
\newblock {\em Combinatorics, Probability and Computing}, 21(1-2):57--87, 2012.

\bibitem[Bon71]{bondy1971pancyclic}
JA~Bondy.
\newblock Pancyclic graphs i.
\newblock {\em Journal of Combinatorial Theory, Series B}, 11(1):80--84, 1971.

\bibitem[BV04]{KKT}
S.~Boyd and L.~Vandenberghe.
\newblock {\em Convex optimization}.
\newblock Cambridge University Press, Cambridge, 2004.

\bibitem[CGP97]{clark1997number}
LH~Clark, JC~George, and TD~Porter.
\newblock On the number of l-factors in the n-cube.
\newblock {\em Congressus Numerantium}, pages 67--70, 1997.

\bibitem[CL75]{lorimer}
E.~J. Cockayne and P.~J. Lorimer.
\newblock The {R}amsey number for stripes.
\newblock {\em J. Austral. Math. Soc.}, 19:252--256, 1975.

\bibitem[DJ16]{DayJoh}
A.~N. {Day} and J.~R. {Johnson}.
\newblock {Multicolour Ramsey Numbers of Odd Cycles}.
\newblock {\em ArXiv e-prints}, February 2016, 1602.07607.

\bibitem[EFRS76]{Erd}
P.~Erd{\H{o}}s, R.~J. Faudree, C.~C. Rousseau, and R.~H. Schelp.
\newblock Generalized ramsey theory for multiple colors.
\newblock {\em Journal of Combinatorial Theory, Series B}, 20(3):250--264,
  1976.

\bibitem[EG59]{EG}
P.~Erd{\H{o}}s and T.~Gallai.
\newblock On maximal paths and circuits of graphs.
\newblock {\em Acta Math. Acad. Sci. Hungar}, 10:337--356 (unbound insert),
  1959.

\bibitem[EG75]{ErdGra}
P.~Erd{\H{o}}s and R.~L. Graham.
\newblock On partition theorems for finite graphs.
\newblock In {\em Infinite and finite sets ({C}olloq., {K}eszthely, 1973;
  dedicated to {P}. {E}rd{\H o}s on his 60th birthday), {V}ol. {I}}, pages
  515--527. Colloq. Math. Soc. J\'anos Bolyai, Vol. 10. North-Holland,
  Amsterdam, 1975.

\bibitem[Erd47]{ErdL}
P.~Erd{\H{o}}s.
\newblock Some remarks on the theory of graphs.
\newblock {\em Bulletin of the American Mathematical Society}, 53(4):292--294,
  1947.

\bibitem[ES35]{ErdSze}
P.~Erd{\H{o}}s and G.~Szekeres.
\newblock A combinatorial problem in geometry.
\newblock {\em Compositio Mathematica}, 2:463--470, 1935.

\bibitem[Fer15a]{ferguson2015ramsey1}
D.~Ferguson.
\newblock The ramsey number of mixed-parity cycles i.
\newblock {\em arXiv preprint arXiv:1508.07154}, 2015.

\bibitem[Fer15b]{ferguson2015ramsey2}
D.~Ferguson.
\newblock The ramsey number of mixed-parity cycles ii.
\newblock {\em arXiv preprint arXiv:1508.07171}, 2015.

\bibitem[Fer15c]{ferguson2015ramsey3}
D.~Ferguson.
\newblock The ramsey number of mixed-parity cycles iii.
\newblock {\em arXiv preprint arXiv:1508.07176}, 2015.

\bibitem[F{\L}07a]{FigLucz}
A.~Figaj and T.~{\L}uczak.
\newblock The {R}amsey number for a triple of long even cycles.
\newblock {\em J. Combin. Theory Ser. B}, 97(4):584--596, 2007.

\bibitem[F{\L}07b]{FigLucz2}
A.~{Figaj} and T.~{{\L}uczak}.
\newblock {The Ramsey number for a triple of large cycles}.
\newblock {\em ArXiv e-prints}, September 2007, 0709.0048.

\bibitem[FS74]{FauSch}
R.~J. Faudree and R.~H. Schelp.
\newblock All {R}amsey numbers for cycles in graphs.
\newblock {\em Discrete Math.}, 8:313--329, 1974.

\bibitem[GH88]{GraHar}
N.~Graham and F.~Harary.
\newblock The number of perfect matchings in a hypercube.
\newblock {\em Applied Mathematics Letters}, 1(1):45--48, 1988.

\bibitem[GRSS07]{Gya}
A.~Gy\'arf\'as, M.~Ruszink\'o, G.~N. S\'ark\"ozi, and E.~Szemer\'edi.
\newblock Three-color {R}amsey numbers for paths.
\newblock {\em Combinatorica}, 27(1):35--69, 2007.

\bibitem[KS96]{KomSim}
J.~Koml{\'o}s and M.~Simonovits.
\newblock Szemer\'edi's regularity lemma and its applications in graph theory.
\newblock In {\em Combinatorics, Paul Erd\H os is eighty, Vol.\ 2 (Keszthely,
  1993)}, volume~2 of {\em Bolyai Soc. Math. Stud.}, pages 295--352. J\'anos
  Bolyai Math. Soc., Budapest, 1996.

\bibitem[KSS05]{Exact}
Y.~Kohayakawa, M.~Simonovits, and J.~Skokan.
\newblock The 3-colored {R}amsey number of odd cycles.
\newblock In {\em Proceedings of GRACO2005}, volume~19 of {\em Electron. Notes
  Discrete Math.}, pages 397--402 (electronic), Amsterdam, 2005. Elsevier.

\bibitem[{\L}SS12]{Multi}
T.~{\L}uczak, M.~Simonovits, and J.~Skokan.
\newblock On the multi-colored {R}amsey numbers of cycles.
\newblock {\em J. Graph Theory}, 69(2):169--175, 2012.

\bibitem[{\L}uc99]{Lucz}
T.~{\L}uczak.
\newblock {$R(C\sb n,C\sb n,C\sb n)\leq(4+o(1))n$}.
\newblock {\em J. Combin. Theory Ser. B}, 75(2):174--187, 1999.

\bibitem[{\"{O}}P13]{OstPet}
P.~RJ {\"{O}}sterg{\aa}rd and V.~H. Pettersson.
\newblock Enumerating perfect matchings in n-cubes.
\newblock {\em Order}, pages 1--15, 2013.

\bibitem[PL86]{plummer1986matching}
Michael~D Plummer and L{\'a}szl{\'o} Lov{\'a}sz.
\newblock {\em Matching theory}, volume~29.
\newblock Elsevier, 1986.

\bibitem[Rad94]{survey}
S.~P. Radziszowski.
\newblock Small {R}amsey numbers.
\newblock {\em Electron. J. Combin.}, 1:Dynamic Survey 1, 30 pp. (electronic),
  1994.

\bibitem[Ram30]{Ram}
F.~P. Ramsey.
\newblock On a {P}roblem of {F}ormal {L}ogic.
\newblock {\em Proc. London Math. Soc.}, S2-30(1):264, 1930.

\bibitem[Ros73]{Rosta}
V.~Rosta.
\newblock On a {R}amsey-type problem of {J}. {A}. {B}ondy and {P}. {E}rd{\H
  o}s. {I}, {II}.
\newblock {\em J. Combinatorial Theory Ser. B}, 15:94--104; ibid. 15 (1973),
  105--120, 1973.

\bibitem[Rud76]{rudin}
W.~Rudin.
\newblock {\em Principles of Mathematical Analysis}.
\newblock International series in pure and applied mathematics. McGraw-Hill,
  1976.

\bibitem[Sze78]{Szem}
E.~Szemer{\'e}di.
\newblock Regular partitions of graphs.
\newblock In {\em Probl\`emes combinatoires et th\'eorie des graphes (Colloq.
  Internat. CNRS, Univ. Orsay, Orsay, 1976)}, volume 260 of {\em Colloq.
  Internat. CNRS}, pages 399--401. CNRS, Paris, 1978.

\end{thebibliography}

\appendix
\section{Proof of Lemmas \ref{embed} and \ref{embed2}}

We use the following simple property of regular pairs which appears as Lemma~5 in \cite{FigLucz2}.
\begin{lemma}\label{embe}
Let $1/m\ll\del\ll d$ and let $G=(V_1, V_2)$ be a $(\del, d)$-super-regular pair with $\abs{V_1}=\abs{V_2}=m$. Then for each pair $u\in V_1$, $w\in V_2$, $G$ contains a $uw$-path of length $\ell$ for each odd $3\leq\ell\leq 2(1-5\del)m$.
\end{lemma}

\begin{lem:embed}
Let $q\geq4$ and suppose that $\frac{1}{m}\ll\del\ll d$.  Let $F$ be a connected matching of order $q$ such that every vertex of $F$ is incident to a matching edge and let $H$ be a $(\del,d,m)$-super-regular blow-up of $F$. Then the following holds:  

If $i,j\in V(F)$ and there is an $ij$-path of length $r$ in $F$, then for every  pair of vertices $u\in U_i$, $w\in U_j$, there exists a $uw$-path of length $\ell$ in $H$ for each $3q\leq\ell\leq (1-6\del)qm$ such that $\ell\equiv r\pmod{2}$. 
\end{lem:embed}

\begin{proof}
Take $i,j\in V(F)$ and let $u\in U_i$, $w\in U_j$. Let $T$ be a spanning tree of $F$ which includes every matching edge of $F$. Note that $T$ contains a closed walk $W=y_0 \ldots y_p$, where $y_1=y_p=i$ and $W$ covers each edge of $T$ exactly twice, in particular $p=2(q-1)$ (note that $q=v(F)$). Using basic properties of regular pairs we can find a path $\widetilde{W}=w_0 \ldots w_p$ in $H$ where $u=w_0$ and $w_t\in U_{y_t}$ for all $t$. Let $P=x_0\ldots x_{r}$ be a path of length $r$ in $F$ where $x_0=i, x_{r}=j$. Again, using basic properties of regular pairs we can find a path $\widetilde{P}=v_0 \ldots v_{r}$ in $H$ where $v_0=w_p$, $v_{r}=w$, $v_t\in U_{x_t}$ for all $t$ and $\widetilde{P}$ intersects $\widetilde{W}$ only in the vertex $w_p$. Letting $Q=\widetilde{W}\widetilde{P}$, it follows that $Q$ is a $uw$-path in $H$ of length $r+p=r+2(q-1)\equiv r\pmod{2}$. Suppose that $\{a, b\}$ is a matching edge of $F$ so that  $(U_a, U_b)$ is $(\del, d)$-super-regular in $H$. Note that $Q$ visits each set $U_i$ in $H$ at most 3 times and so there exist $U'_a\subseteq U_a\backslash Q$, $U'_b\subseteq U_b\backslash Q$ such that $\abs{U'_a}=\abs{U'_b}=m-3$. Note that $(U'_a,U'_b)$ is certainly $(2\del, d/2)$-super-regular by Fact \ref{slicing}. By construction, we may pick consecutive vertices $w_t, w_{t+1}$ of $\widetilde{W}$ (and hence $Q$) such that $w_t\in U_a, w_{t+1}\in U_b$. By super-regularity we may then pick vertices $u_a\in N(w_{t+1})\cap U'_a$, $u_b\in N(w_{t})\cap U'_b$ such that $\{u_a,u_b\}$ is an edge of $H$. Applying Lemma~\ref{embe} to $(U_a', U_b')$ and vertices $u_a, u_b$, it follows that we can find a $q_tq_{t+1}$-path in $H$ which intersects $Q$ only at its endpoints and we can choose this path to have any odd length $1\leq\ell\leq 2(1-5\del)(m-3)+2$. Note that letting such a path replace the edge $\{q_t, q_{t+1}\}$ in $Q$ does not change the parity of the length of $Q$. Applying the same argument to each matching edge of $F$ we see that $H$ contains $uw$-paths of each length $r+2(q-1)\leq\ell \leq r+2(q-1)+\frac{q}{2}\cdot2(1-6\del)m$ for which $\ell\equiv r\pmod{2}$. The result follows.
\end{proof}

\begin{lem:embed2}
Let $q\geq 4$ and let $\frac{1}{m}\ll\del\ll d$. Let $F$ be an odd connected matching of order $q$ and suppose that $H$ is a $(\del, m)$-regular blow-up of $F$ with minimum density $d$. Then $H$ contains a cycle of length $\ell$ for each odd $3q\leq\ell\leq(1-6\del)qm$
\end{lem:embed2}

\begin{proof}
Since $F$ is non-bipartite it contains an odd cycle $C$. Since the largest matching in $F$ has $q/2$ edges it follows that $\abs{C}\leq q+1$. Let $T\subseteq F$ be a minimal tree that contains every matching edge of $F$. It is easy to show that $T$ must have $<2q$ vertices. Let $W$ be a closed walk in $T$ which traverses each edge of $T$ precisely twice (so in particular $W$ has even length). Since $W$ and $C$ must intersect, we can augment the walk $W$ by $C$ to obtain a closed walk $W'=x_1 \ldots x_p x_1$ in $F$ where $p$ is odd and $p\leq3q$ by the above. Note that by Facts~\ref{slicing} and \ref{super}, we can find $H'\subseteq H$ such that $H'$ is a $(2\del, d/2,(1-\del)m)$-super-regular blowup of $F$. Let $U_j$ denote the vertex class of $H'$ corresponding to the vertex $j$ in $F$ for each $j\in V(F)$. Using basic properties of regular pairs, we can find an odd cycle $D=v_1\ldots v_pv_1$ in $H'$ where $v_j\in U_{x_j}$ for all $j$. 

Suppose that $\{a, b\}$ is a matching edge of $F$ so that  $(U_a, U_b)$ is $(2\del, d/2)$-super-regular. By construction, we may pick consecutive vertices $v_t, v_{t+1}$ of $D$ such that $v_t\in U_a, v_{t+1}\in U_b$. Note that $D$ visits each set $U_i$ in $H$ at most 3 times. We may therefore apply Lemma~\ref{embe} as we did in the proof of Lemma \ref{embed} to find a $v_tv_{t+1}$-path $Q$ in $H'$ such that $Q$ intersects $D$ only at its endpoints and we can choose $Q$ to have any odd length $1\leq\ell\leq 2(1-5\del)[(1-\del)m-3]+2$. Applying the same argument to each matching edge of $F$ we see that $H$ contains an odd cycle of each odd length $p\leq\ell \leq p+\frac{q}{2}\cdot2(1-6\del)m$. The result follows.

\end{proof}
\end{document}